\documentclass[preprint,12pt]{elsarticle}

\usepackage{proof}
\usepackage{amssymb,amsmath,amsthm}
\usepackage{times}
\usepackage{color}
\usepackage{latexsym}

\usepackage{utfsym}
\newcommand{\crossmark}{\scalebox{0.85}{\usym{2613}}}

\usepackage{stmaryrd}

\usepackage{amssymb}
\usepackage{amsmath}
\usepackage{bussproofs}
\usepackage{prooftree}
\usepackage{diagrams2}
\usepackage{url}

\usepackage{graphicx}

\newtheorem{remark}{Remark}

\newcommand{\CL}{{\sf CL}}
\newcommand{\IL}{{\sf IL}}
\newcommand{\CLL}{{\sf CLL}}
\newcommand{\ILL}{{\sf ILL}}
\newcommand{\ILZ}{{\sf ILL}_\bot}
\newcommand{\PRO}{{\sf PRO}}
\newcommand{\DNE}{{\sf DNE}}

\newcommand{\CLb}{{\sf CL}_{\rm b}}
\newcommand{\ILb}{{\sf IL}_{\rm b}}
\newcommand{\CLLb}{{\sf CLL}_{\rm b}}

\newcommand{\cwedge}{\otimes}
\newcommand{\cvee}{\operatorname{\bindnasrepma}}
\newcommand{\awedge}{\,\&\,}
\newcommand{\bang}[1]{! #1}
\newcommand{\whynot}{?}
\newcommand{\lto}{\multimap}
\newcommand{\lequiv}{\leftrightsquigarrow}

\newcommand{\Id}{{\rm Id}}
\newcommand{\Transform}{{\cal T}}
\newcommand{\IT}[2]{\Transform^{#1}_{#2}}
\newcommand{\ITTr}[1]{\IT{{\rm Tr}}{#1}}

\newcommand{\Trans}[2]{{#1}^{#2}}
\newcommand{\genericTrans}[1]{\Trans{#1}{{\rm Tr}}}
\newcommand{\genericTransPrime}[1]{\Trans{#1}{{\rm Tr}'}}

\newcommand{\TIL}[1]{{#1}^{\dagger}}  
\newcommand{\TCLL}[1]{{#1}^{\ddagger}}  
\newcommand{\Forget}[1]{{#1}^{\mathcal{F}}}

\newcommand{\Goedel}{\textup{G}}
\newcommand{\lGoedel}{\textup{lG}}
\newcommand{\Kuroda}{\textup{Ku}}
\newcommand{\lKuroda}{\textup{lKu}}
\newcommand{\Kolm}[1]{{\textup K}_{#1}}

\newcommand{\Tr}[2]{{#1}_{#2}}
\newcommand{\genericTr}[1]{\Tr{#1}{{\rm Tr}}}

\newcommand{\ruleid}{\text{id}}
\newcommand{\rulecut}{\text{cut}}

\newcommand{\rulecon}{\text{con}}
\newcommand{\rulewkn}{\text{wkn}}

\newcommand{\pdefin}{:\equiv}
\newcommand{\eqleft}[1]{\begin{itemize} \item[] $#1$ \end{itemize}}

\newcommand{\squareOp}{\operatorname{\Box}}

\journal{Annals of Pure and Applied Logic}

\newtheorem{defn}{Definition}
\newtheorem{lem}{Lemma}

\newtheorem{prop}{Proposition}
\newtheorem{thm}{Theorem}

\newtheorem{obs}{Observation}

\begin{document}

\begin{frontmatter}



\title{On the Various Translations between Classical, Intuitionistic and Linear Logic}


 \cortext[cor1]{Corresponding author}
 
\author[label3,label4]{Gilda Ferreira\fnref{label5}\corref{cor1}}
\ead{gmferreira@fc.ul.pt}
\fntext[label5]{The author acknowledges the support of Fundação para a Ciência e a Tecnologia under the projects [UID/04561/2025 and UID/00408/2025] and is also grateful to LASIGE - Computer Science and Engineering Research Centre (Universidade de Lisboa).}
\affiliation[label3]{organization={DCeT},
	addressline={Universidade Aberta}, 
	postcode={1269-001}, 
	state={Lisboa},
	country={Portugal}}

\affiliation[label4]{organization={Center for Mathematical Studies, University of Lisbon (CEMS.UL)},
	addressline={Universidade de Lisboa}, 
	postcode={1749-016}, 
	state={Lisboa},
	country={Portugal}}

\author[label6]{Paulo Oliva\fnref{}}
\ead{p.oliva@qmul.ac.uk}
\affiliation[label6]{
Organization
	addressline={School of Electronic Engineering and Computer Science,}, 
	         city={Queen Mary University of London},
	         postcode={London E1 4NS}, 
	country={United Kingdom}
}
\author[label1]{Clarence Lewis Protin\fnref{label2}}
\ead{cprotin@sapo.pt}
\affiliation[label1]{
       organization={Centro de Filosofia},
           addressline={Faculdade de Letras da Universidade de Lisboa}, 
            country={Portugal}
}
\fntext[label2]{This work is financed by national funds through FCT – Fundação para a Ciência e a Tecnologia, I.P., within the scope of the project REF with the identifier DOI 10.54499/UIDB/00310/2020.}
        

\begin{abstract}
Several different proof translations exist between classical and intuitionistic logic (negative translations), and intuitionistic and linear logic (Girard translations). Our aims in this paper are (1) to consider extensions of intuitionistic linear logic which correspond to each of these systems, and (2) with this common logical basis, to develop a uniform approach to devising and simplifying proof translations. As we shall see, through this process of ``simplification'' we obtain most of the well-known translations in the literature.
\end{abstract}





\end{frontmatter}




\section{Introduction}
\label{sec-introduction}

\emph{Classical logic}, the most widely known and studied form of mathematical reasoning, is based on the paradigm of \emph{truth}. During the 20th century some mathematicians started focusing on \emph{justification} rather than truth. \emph{Intuitionistic logic} \cite{heyting,mints} captures this emphasis on justification, dealing with the constructive nature of proofs. The intuitionistic connectives, instead of being seen as truth functions  propagating the logical values of \emph{true} and \emph{false}, are seen as generating new proofs from previously constructed proofs.

In the Gentzen sequent calculus \cite{gen1,gentzen}, intuitionistic logic can be obtained from classical logic by restricting the number of formulas on the right side of sequents to at most one formula. In this way, one can view intuitionistic reasoning as a restricted form or classical reasoning.

Linear logic \cite{GirTCS,lec} can also be seen as trying to formalise ``constructivism'', by imposing explicit control over the uses of structural rules such as weakening and contraction. One could say that instead of focusing on truth or proofs, linear logic places its emphasis on the proof \emph{resources}, and the manner in which these resources are ``used'' during an inference.


In the literature we find several translations between
classical and intuitionistic logic, and various translations involving linear logic. The translations from classical to intuitionistic logic are often called \emph{negative translations} or \emph{double-negation translations}, the most well-known being translations due to Glivenko \cite{sur}, G\"odel and Gentzen \cite{gen1, god}, Kuroda \cite{kur}, and Krivine \cite{kri}. For comprehensive surveys on negative tranlsations see \cite{ol}, \cite[Chapter 6]{lecture}, \cite[Chapter 2.3]{cons} and \cite{basic}. For further developments and applications of negative translations in different contexts see \cite{Dosen, Litak} (modal logic), \cite{Far} (substructural logics), \cite{berger, kohlenbach} (extraction of computational content from classical proofs), \cite{Ecumenic} (ecumenical systems), and \cite{boudard,lau2} (polarization).

The translations involving linear logic are often called \emph{Girard translations}. They can be classified into three categories: (1) translations of classical linear logic into intuitionistic linear logic, which can be seen as substructural variants of the negative translations \cite{CCP03, Far,kan,Lam95,Lau18,MT10}, (2) translations of intuitionistic logic into linear logic \cite{DJS95b, CFMM16, lec, Sch91}, and (3) translations of classical logic into classical linear logic \cite{cosmo, DJS95a,DJS97}.

The first contribution of the present paper is an analysis and comparison of these different logical systems via \emph{Intuitionistic Linear Logic} $\ILL$. Reformulating Classical Logic $\CL$, Intuitionistic Logic
$\IL$ and Classical Linear Logic $\CLL$ in this common framework allows us to pinpoint precisely the differences between these systems and provide natural ways of looking for translations between them. 

Our second contribution builds on the work of the first two authors \cite{ol, ol2} on the various \emph{negative translations}, and how the G\"{o}del-Gentzen \cite{gen1,god} and Kuroda \cite{kur} translations can be obtained as ``\emph{simplifications}'' of the Kolmogorov \cite{kol} translation. The work in \cite{ol, ol2} is based on inner and outer simplifications from Kolmogorov's translation. However, that approach did not distinguish between inner and outer presentations and was centred on single negations. In the present work, we refine this perspective by explicitly considering inner and outer presentations, and now we treat $\neg \neg$ as a single transformation. Additionally, we extend the new framework to also include the $A \mapsto \; !A$ transform, which will allow us study translations from $\IL$ into $\ILL$ (just sketched in \cite{ol2} in the previous framework) as well as translations from $\CLL$ to $\ILL$, and from $\CL$ to $\CLL$ (viewed as compositions of $\CL$ to $\IL$, and $\IL$ to $\ILL$ translations). 





\subsection{$\ILL$ as a common base system}

\begin{table}[t]
	\[
	\begin{array}{|cccc|}
		\hline
		& & & \\
		\multicolumn{4}{|c|}{
			\begin{array}{ccc}
                \begin{prooftree} 
                    \justifies 
                    A \vdash A 
                    \using(\ruleid) 
                \end{prooftree}
                & \quad\quad &
                \begin{prooftree} 
                    \Gamma \vdash A \quad \Delta, A \vdash C
        			\justifies \Gamma, \Delta \vdash C 
                    \using (\rulecut)
        		\end{prooftree} \\[5mm]
            \end{array}
        } \\[5mm]
		\multicolumn{4}{|c|}{
			\begin{array}{cccc}
                \begin{prooftree}
                    \Gamma \vdash C 
                    \justifies 
                    \Gamma, 1 \vdash C 
                    \using (1\textup{L})
                \end{prooftree}  
                \quad
                & 
                \begin{prooftree}
                    \justifies \vdash 1 
                    \using (1\textup{R})
                \end{prooftree} 
                \quad
                &
                \begin{prooftree}
                    \justifies 
                    \Gamma \vdash \top
                    \using (\top\textup{R})
                \end{prooftree} 
                \quad
                &
        		\begin{prooftree} 
                    \justifies 
                    \Gamma, 0 \vdash C  
                    \using (0\textup{L})
                \end{prooftree} 
            \end{array}
        } \\[7mm]
		\hline
		& & & \\
		\multicolumn{4}{|c|}{
			\begin{array}{ccc}
                \begin{prooftree} 
                    \Gamma \vdash A \quad \Delta \vdash B
        			\justifies 
                    \Gamma,\Delta \vdash A \cwedge B 
                    \using(\cwedge\textup{R}) 
                \end{prooftree} 
                & \quad\quad &
        		\begin{prooftree} 
                    \Gamma, A, B \vdash C 
                    \justifies \Gamma, A \cwedge B \vdash C 
                    \using (\cwedge\textup{L}) 
                \end{prooftree} \\[7mm]
                \begin{prooftree} 
                    \Gamma, A \vdash B 
                    \justifies \Gamma \vdash A \lto B 
                    \using (\lto\textup{R}) 
                \end{prooftree} & &
        		\begin{prooftree} 
                    \Gamma \vdash A \quad \Delta, B \vdash C 
                    \justifies 
                    \Gamma, \Delta, A \lto B \vdash C 
                    \using (\lto\textup{L}) 
                \end{prooftree} \\[7mm]
            \end{array}
        } \\[5mm]
		\multicolumn{4}{|c|}{ \quad
			\begin{array}{ccccc}
				\begin{prooftree} \Gamma \vdash A \quad \Gamma \vdash B\justifies  \Gamma \vdash A \awedge B \using (\awedge\textup{R}) \end{prooftree}
				& \quad &
				\begin{prooftree} \Gamma, A \vdash C\justifies  \Gamma, A\awedge  B \vdash C \using (\awedge\textup{L}) \end{prooftree}
				& \quad &
				\begin{prooftree} \Gamma, B \vdash C \justifies \Gamma, A \awedge B \vdash C \using (\awedge\textup{L}) \end{prooftree} \\[7mm]
				\begin{prooftree} \Gamma \vdash A\justifies  \Gamma\vdash A \oplus B \using (\oplus\textup{R}) \end{prooftree}
				& &
				\begin{prooftree} \Gamma \vdash B\justifies  \Gamma\vdash A \oplus B \using (\oplus\textup{R}) \end{prooftree}
				& &
				\begin{prooftree} \Gamma, A \vdash C \quad \Gamma, B \vdash C\justifies  \Gamma, A \oplus B \vdash C \using (\oplus\textup{L})  \quad\end{prooftree} \\[7mm]
			\end{array}
		} \\
		\hline
	\end{array}
	\]
	\caption{Intuitionistic Linear Logic $\ILL$ (constants and connectives)} \label{ill}
\end{table}

\begin{table}[t]
	\[
	\begin{array}{|rccc|}
		\hline
		\hspace{15mm} & & & \\
		& \begin{prooftree} \Gamma \vdash A \justifies \Gamma \vdash \forall x A \using (\forall \textup{R}) \end{prooftree}
		& &
		\begin{prooftree} \Gamma, A[t/x] \vdash C \justifies \Gamma, \forall x A \vdash C \using (\forall \textup{L}) \end{prooftree} \\[5mm]
		& \begin{prooftree} \Gamma \vdash A[t/x] \justifies \Gamma \vdash \exists x A \using (\exists \textup{R}) \end{prooftree}
		& &
		\begin{prooftree} \Gamma, A \vdash C \justifies \Gamma, \exists x A \vdash C \using (\exists \textup{L}) \end{prooftree} \\[5mm]
		\hline
		& & & \\
		\multicolumn{4}{|l|}{
			\quad
			\begin{prooftree} \Gamma, \bang A, \bang A \vdash C \justifies \Gamma, \bang A \vdash C \using (\rulecon) \end{prooftree}
			\quad \quad\quad
			\begin{prooftree} \Gamma \vdash C \justifies \Gamma, \bang A \vdash C \using (\rulewkn) \end{prooftree}
			\quad \quad
			\begin{prooftree} \bang \Gamma \vdash A \justifies \bang \Gamma \vdash \bang A \using (\bang\textup{R}) \end{prooftree}
			\quad \quad
			\begin{prooftree} \Gamma, A \vdash C \justifies  \Gamma, \bang A \vdash C \using (\bang\textup{L}) \quad \end{prooftree}
		} \\
		& & & \\
		\hline
	\end{array}
	\]
	\caption{Intuitionistic Linear Logic (quantifiers and modality)} \label{ill-quantifiers}
\end{table}

Tables \ref{ill} and \ref{ill-quantifiers} describe Intuitionistic Linear Logic ($\ILL$), which we will take as the base system in our study. In $\ILL$ we have three logical constants $\top$, $0$, $1$, four logical connectives $\cwedge$, $\awedge$, $\oplus$, $\lto$, the universal and existential quantifiers $\forall$, $\exists$, and the exponential (or modality) $\bang A$. Sequents are of the form $\Gamma \vdash A$, where we have exactly one formula $A$ on the right-hand side, and the context $\Gamma$ is a multi-set of formulas\footnote{In all rules where a formula $A$ is removed from the context (e.g. $\lto$\textup{R}), only one occurrence of $A$ is removed.}. 

In \cite[page 20]{lec}, Troestra describes an extension of intuitionistic linear logic $\ILL$ with one more logical constant (which we will denote by $\bot$, but Troestra calls $0$) and the following axiom and rule:
\[
\begin{array}{lccr}
	\begin{prooftree}
        \justifies 
        \bot \vdash
        \using (\bot\textup{L})
    \end{prooftree} 
    &~ &~\hspace{2cm} &
    \begin{prooftree} 
        \Gamma \vdash 
        \justifies 
        \Gamma\vdash \bot
        \using (\bot\textup{R})
	\end{prooftree}
\end{array}
\]
We call this extension of intuitionistic linear logic $\ILZ$. In $\ILZ$ 
sequents are either of the form $\Gamma \vdash A$ or $\Gamma \vdash \,$ (i.e. they have at most one formula on the right), and the context $\Gamma$ is a multi-set of formulas. \\[-2mm]

\noindent {\bf Notation}. We will write $A \lequiv B$ as an abbreviation for $(A \lto B) \awedge (B \lto A)$. Negation $\neg A$ is an abbreviation for $A \lto \bot$, and will only be used when considering extensions of $\ILZ$. \\[-2mm]

Our study will take place solely in the language of $\ILL$ or $\ILZ$. Since the languages of Intuitionistic Logic $\IL$, Classical Logic $\CL$, and Classical Linear Logic $\CLL$ are different from that of $\ILL$, we will be working with extensions of $\ILL$ and $\ILZ$ into which we will be able to translate $\IL, \CL$ and $\CLL$. These extensions will include combinations of the following two axiom schemas:
\begin{itemize}
	\item[$(\PRO)$] $A \vdash \, \bang A$
	\item[$(\DNE)$] $\neg \neg A \vdash A$.
\end{itemize}
In this paper, an \emph{extension of $\ILL$ or $\ILZ$} is simply a logic obtained by adding new axioms -- but within the same language. In order to deal with logics that have a language different from $\ILL$, for instance $\CL$ or $\IL$, we will consider translations that map the other language into the language of $\ILL$.

\subsection{Intuitionistic Logic $\IL$ as $\ILL + \PRO$}

\newcommand{\Lang}[1]{{\mathcal L}(#1)}

Let us write $\Lang{\IL}$ and $\Lang{\ILL}$, for the languages of $\IL$ and $\ILL$ respectively. Consider the following translation of formulas of $\IL$ into formulas of $\ILL$:

\begin{defn}[Translating $\Lang{\IL}$ into $\Lang{\ILL}$] \label{dagger-trans} For each formula $A$ of $\Lang{\IL}$ define a formula $\TIL{A}$, in the language of $\ILL$, as follows:
\[
\begin{array}{rclrcl}
    \TIL{(A \wedge B)} & \pdefin &  \TIL{A} \awedge \TIL{B} &
    \TIL{P} & \pdefin & P, \textup{~for $P$ atomic except $\bot$} \\[1mm] \TIL{(A \vee B)} & \pdefin &  \TIL{A} \oplus \TIL{B} \quad & \TIL{\bot} & \pdefin & 0     \\[1mm]  \TIL{(A \rightarrow B)} & \pdefin &  \TIL{A} \lto \TIL{B}  \quad & \TIL{(\forall x A)} & \pdefin & \forall x \TIL{A} \\[1mm]
     &  &  &
    \TIL{(\exists x A)} & \pdefin &  \exists x \TIL{A}.
\end{array}
\]
\end{defn}

It is easy to see that this translation of formulas extends to a translation of proofs, from proofs in $\IL$ to proofs in $\ILL + \PRO$. First, we need a simple lemma:

\begin{lem} \label{lemma-IL} Let $\ILb := \ILL + \PRO$. The following are provable in $\ILb$:
	\begin{itemize}
		\item[$(i)$] $A \lequiv \, \bang A$
		\item[$(ii)$] $A\awedge B \lequiv A\cwedge B$
		\item[$(iii)$] $\top \lequiv 1$.
	\end{itemize}
\end{lem}
\begin{proof} $(i)$ In $\ILL$ we have $\bang A \vdash A$, a consequence of $(\bang\textup{L})$. $\PRO$ gives us the converse $A \vdash \, \bang A$. $(ii)$ Since we have $\bang (A \awedge B) \lequiv \; \bang A \cwedge \bang B$ the result follows from $(i)$. $(iii)$ One possible instance for the axiom schema $(\top\textup{R})$ is $1 \vdash \top$. On the other hand, we have $\vdash 1$ $(1\textup{R})$ and hence $\bang\top \vdash 1$, by weakening. By $(i)$, we get $\top \vdash 1$. \end{proof}

Then it follows that $\IL$ can be seen as an extension of $\ILL$ with the promotion axiom $\PRO$. 

\begin{prop}[$\IL$ as an extension of $\ILL$] \label{prop-IL} $\vdash_{\IL} A$ iff $\vdash_{\ILb} \TIL{A}$.
\end{prop}
\begin{proof} From left to right, we can use Girard's translation (see \cite{GirTCS}) of $\IL$ into $\ILL$, together with the observation of Lemma \ref{lemma-IL} $(i)$ that in $\ILL + \PRO$ we have $A\lequiv ~\bang A$, so all the $!$'s in Girard's translation can be omitted. The implication from right to left can be obtained by noting that the (forgetful) inverse $\Forget{(\cdot)} \colon \ILL \mapsto \IL$ of $\TIL{(\cdot)} \colon \IL \mapsto \ILL$, which, for example, maps $0$ to $\bot$, $\bang A$ to $A$, $A \lto B$ back to $A \to B$, and both additive and multiplicative conjunctions to the $\IL$-conjunction $\wedge $, makes all the axioms and rules of $\ILL + \PRO$ derivable in $\IL$, and $\Forget{(\TIL{A})} \equiv A$.
\end{proof}

\begin{remark} The $\IL$ derivation of the formula $\Forget{A}$ (defined in the proof above) is what has been called (see \cite{DJS95a}) the \emph{skeleton} of the original $\ILL$ derivation of $A$. As observed in \cite{DJS95b}, ``the skeleton of a derivation that has been obtained from the Girard translation of an $\IL$ derivation, will not always be the $\IL$ derivation that we started with'', and they study optimal proof-by-proof embeddings of $\IL$ into $\ILL$. In our present paper, however, we will be focusing on the formula translations, and how these can be simplified, without worrying about the optimality of the corresponding proof translations -- which we hope to investigate in a future work.
\end{remark}

\subsection{Classical Linear Logic $\CLL$ as $\ILZ + \DNE$}

As in the previous section, we can also translate the language of $\CLL$ into the language of $\ILZ$ as follows:

\begin{defn}[Translating $\Lang{\CLL}$ into $\Lang{\ILZ}$] \label{ddagger-trans} The translation $A \mapsto \TCLL{A}$ is defined inductively as follows:
\[
\begin{array}{rclrcl}
   
    \TCLL{(A \otimes B)} & \pdefin &  \TCLL{A} \otimes \TCLL{B}    &\TCLL{P} & \pdefin & P \quad (\textup{$P$ atomic})\\[1mm]
    \TCLL{(A \awedge B)} & \pdefin &  \TCLL{A} \awedge \TCLL{B} \quad &     \TCLL{(\forall x A)} & \pdefin &  \forall x \TCLL{A}\\[1mm]
    \TCLL{(A \oplus B)} & \pdefin &  \TCLL{A} \oplus \TCLL{B} \quad & \TCLL{(\exists x A)} & \pdefin &  \exists x \TCLL{A}\\[1mm]
     \TCLL{(A \lto B)} & \pdefin & \TCLL{A} \lto \TCLL{B}  \quad & \TCLL{(\bang A)} & \pdefin &   \bang  \TCLL{A} \\[1mm]
     
      \TCLL{(A \cvee B)} & \pdefin &  \neg (\neg \TCLL{A} \otimes \neg \TCLL{B})     \quad & \TCLL{(\whynot A)} & \pdefin &  \neg \bang \neg \TCLL{A}\\[1mm]
\end{array}
\]
\end{defn}


As in the case of Intuitionistic Logic (Proposition \ref{prop-IL}), it is easy to check that the above translation of formulas can be lifted to a translation of $\CLL$ proofs into $\ILZ + \DNE$ proofs.

\begin{prop}[$\CLL$ as an extension of $\ILZ$] \label{prop-CLL} Let $\CLLb := \ILZ + \DNE$. Then $\vdash_{\CLL} A$ iff $\vdash_{\CLLb} \TCLL{A}$.
\end{prop}
\begin{proof} From left-to-right the result follows by induction on the derivation of $A$ in $\CLL$. For the converse, note that $\CLL$ proves the equivalences $\neg (\whynot A) \lequiv~\bang \neg A$ and $\neg (A \cvee B) \lequiv \neg A \otimes \neg B$.
\end{proof}


\subsection{Classical Logic $\CL$ as $\ILZ + \PRO + \DNE$ (or $\ILL + \PRO + \DNE$)}

We can view classical logic $\CL$ as $\ILZ + \PRO + \DNE$ or $\ILL + \PRO + \DNE$, since in the presence of $\PRO$ and $\DNE$ we can prove $0 \lequiv \bot$. In this way, we can say that Intuitionistic and Classical Logic share the same language, and we can make use of the same translation $\TIL{(\cdot)}$ -- cf. Definition \ref{dagger-trans} -- to translate $\CL$ proofs into $\ILZ + \PRO + \DNE$ proofs, taking $\TIL{\bot} = \bot$. We then have the equivalent of Proposition \ref{prop-IL} for Classical Logic:

\begin{prop}[$\CL$ as an extension of $\ILL$] \label{CL} Let $\CLb := \ILZ + \PRO + \DNE$. Then $\vdash_{\CL} A$ iff $\vdash_{\CLb} \TIL{A}$.
\end{prop}
\begin{proof} It is well known that Classical Logic $\CL$ can be formalized as Intuitionistic Logic $\IL$ plus double negation elimination $\neg \neg B \to B$, where $\neg B$ here stands for the standard intuitionistic negation. Thus, $\vdash_{\CL} A $ iff $\vdash_{\IL + (\neg \neg B \rightarrow B)} A$. Since, $\TIL{(\neg \neg B \to B)} = \DNE$, by Proposition \ref{prop-IL} we get  $\vdash_{\CL} A$
iff $\vdash_{\ILZ + \PRO + \DNE} \TIL{A}$.
\end{proof}

\subsection{On some equivalences}

From now on we will work with $\ILL$ and $\ILZ$ and the following extensions 
\begin{align*}
    \ILb & := \ILL + \PRO \\
    \IL_\bot & := \ILZ + \PRO \\
    \CLLb & := \ILZ + \DNE \\
    \CLb & := \ILZ + \PRO + \DNE
\end{align*}
As we have seen, $\ILb, \CLLb$ and $\CLb$ are the projections of Intuitionistic Logic $\IL$, Classical Linear Logic $\CLL$, and Classical Logic $\CL$ in the language of $\ILL$ or $\ILZ$. This means that we are working entirely in the language of $\ILL$ or $\ILZ$, and when we speak of the validity or provability of an $\ILL$-formula $A$ in $\ILb, \CLLb$ or $\CLb$, we mean the validity or provability of $A$ in the corresponding $\ILL$-extension.

Our analysis of the various proof translations will rely on the validity or failure of certain equivalences. Let us start with the equivalences which are valid in $\CLb$, but hold or fail in $\IL_\bot$ and $\ILZ$.

\begin{prop} \label{not-not-equivalences} The following $\CLLb$ (and hence also $\CLb$) equivalences hold / fail in $\IL_\bot$ and $\ILZ$:
\[
\begin{array}{lrclcc}
    & & & & \quad \IL_\bot \quad & \ILZ \\[1mm]
    \hline
    (i) & \neg \neg (\neg \neg A \cwedge \neg \neg B) 
    & \lequiv &
    \neg \neg (A \cwedge B) 
        & \checkmark & \checkmark \\[1mm]
    (ii) & \neg \neg (\neg \neg A \awedge \neg \neg B) 
       & \lequiv &
      \neg \neg (A \awedge B) 
      & \checkmark & \crossmark \\[1mm]
    (iii) & \neg \neg (\neg \neg A \oplus \neg \neg B) 
    & \lequiv &
    \neg \neg (A \oplus B)  
      & \checkmark & \checkmark \\[1mm]
    (iv) & \neg \neg (\neg \neg A \lto \neg \neg B) 
    & \lequiv &
    \neg \neg (A \lto B)  
      & \checkmark & \crossmark \\[1mm]
    (v) & \neg \neg (\neg \neg A \lto \neg \neg B) 
    & \lequiv &
    \neg \neg (A \lto \neg \neg B)  
      & \checkmark & \checkmark \\[1mm]
    (vi) & \neg \neg \forall x \neg \neg A 
    & \lequiv &
    \neg \neg \forall x A 
      & \crossmark & \crossmark \\[1mm]
    (vii) & \neg \neg \exists x \neg \neg A 
    & \lequiv &
    \neg \neg \exists x A 
      & \checkmark & \checkmark \\[1mm]
    (viii) & \neg \neg \bang \neg \neg A 
        & \lequiv &
        \neg \neg \bang A 
      & \checkmark & \crossmark \\[1mm]
    \hline
    (ix) & \neg \neg (\neg \neg A \cwedge \neg \neg B) 
    & \lequiv &
    \neg \neg A \cwedge \neg \neg B 
        & \checkmark & \crossmark \\[1mm]
    (x) & \neg \neg (\neg \neg A \awedge \neg \neg B) 
    & \lequiv &
    \neg \neg A \awedge \neg \neg B 
        & \checkmark & \checkmark \\[1mm]
    (xi) & \neg \neg (\neg \neg A \oplus \neg \neg B) 
    & \lequiv &
    \neg \neg A \oplus \neg \neg B 
        & \crossmark & \crossmark \\[1mm]
    (xii) & \neg \neg (\neg \neg A \lto \neg \neg B) 
    & \lequiv &
    \neg \neg A \lto \neg \neg B 
        & \checkmark & \checkmark \\[1mm]
    (xiii) & \neg \neg \forall x \neg \neg A 
    & \lequiv &
    \forall x \neg \neg A 
        & \checkmark & \checkmark \\[1mm]
    (xiv) & \neg \neg \exists x \neg \neg A 
    & \lequiv &
    \exists x \neg \neg A 
        & \crossmark & \crossmark \\[1mm]
    (xv) & \neg \neg \bang \neg \neg A 
    & \lequiv &
    \bang \neg \neg A 
        & \checkmark & \crossmark \\[1mm]
\end{array}
\]
\end{prop}
\begin{proof} Clearly, all these equivalences hold in the presence of $\DNE$. Positive results (provability in $\IL_\bot = \ILZ + \PRO$ and/or $\ILZ$) are easy exercises. The fact that the equivalences in $(vi)$, $(xi)$ and $(xiv)$ fail in $\IL_\bot$ (and hence also in $\ILZ$) is well known and can be shown by simple Kripke structures. One can also show that $(ii)$, $(iv)$, $(viii)$, $(ix)$ and $(xv)$ are not provable in $\ILZ$ by constructing bounded pocrims (partially ordered residuated integral monoids) where these fail (for instance, see \cite{ao} for a pocrim where $(ix)$ fails). 
%
\end{proof}

When looking at translations of $\ILb$ into $\ILL$ we will also need the following results:

\begin{prop} \label{bang-equivalences} The following $\ILb = \ILL + \PRO$ equivalences hold / fail in $\ILL$\footnote{The equivalences valid in $\ILL$ are also valid in $\CLLb$. This result will be used in Section 6.}:
\[
\begin{array}{lrclc}
    & & & & \quad \ILL \quad \\[1mm]
    \hline
    (i) & \bang (\bang A \cwedge \bang B) 
    & \lequiv & 
    \bang A \cwedge \bang B 
        & \checkmark \\[1mm]
    (ii) & \bang (\bang A \awedge \bang B) 
    & \lequiv & 
    \bang A \awedge \bang B 
        & \crossmark \\[1mm]
    (iii) & \bang (\bang A \oplus \bang B) 
    & \lequiv &
    \bang A \oplus \bang B  
        & \checkmark \\[1mm]
    (iv) & \bang (\bang A \lto \, \bang B) 
    & \lequiv &
    \bang A \lto \, \bang B  
        & \crossmark \\[1mm]
    (v) & \bang \forall x \bang A 
    & \lequiv &
    \forall x \bang A 
        & \crossmark \\[1mm]
    (vi) & \bang \exists x \bang A 
    & \lequiv &
    \exists x \bang A 
        & \checkmark \\[1mm]
    (vii) & \bang \bang \bang A 
    & \lequiv &
    \bang \bang A  
        & \checkmark \\[1mm]
    \hline
    (viii) & \bang (\bang A \cwedge \bang B) 
    & \lequiv &
    \bang (A \cwedge B)  
        & \crossmark \\[1mm]
    (ix) & \bang (\bang A \awedge \bang B) 
    & \lequiv &
    \bang (A \awedge B)  
        & \checkmark \\[1mm]
    (x) & \bang (\bang A \lto \, \bang B) 
    & \lequiv &
    \bang (A \lto B)  
        & \crossmark \\[1mm]
    (xi) & \bang (\bang A \lto \, \bang B) 
    & \lequiv &
    \bang (\bang A \lto B)  
        & \checkmark \\[1mm]
    (xii) & \bang \forall x \bang A 
    & \lequiv &
    \bang \forall x A  
        & \checkmark \\[1mm]
    (xiii) & \bang \exists x \bang A 
    & \lequiv &
    \bang \exists x A  
        & \crossmark
\end{array}
\]
\end{prop}
\begin{proof} Clearly, using $\PRO$, all equivalences are valid in $\ILb$. The fact that some equivalences fail in $\ILL$ can be shown by constructing simple counter-models, which can be done using the algebraic semantics of $\ILL$ \cite{lec} and a tool such as mace4 \cite{prover9-mace4}.
\end{proof}

\section{Modular Translations and Simplifications}
\label{sec-modular-translations}

We have seen in the previous section that when we translate $\IL$, $\CL$ and $\CLL$ into the language of $\ILL$ (or $\ILZ$) we obtain the corresponding logics $\ILb$, $\CLb$ and $\CLLb$ -- extensions of $\ILL$ with combinations of the axioms $\PRO$ and $\DNE$. Thus, the systems essentially only differ in their capacity to distinguish between a formula $A$ and its exponential $\bang A$, and a formula $A$ and its double negation $\neg \neg A$. This provides some hints for the construction of translations between these calculi.

Using an arrow from a logic $L_1$ to another logic $L_2$ to indicate that $L_2$ is an extension\footnote{We say that a logic $L_1$ is an extension of a logic $L_2$ if $L_1$ is obtained from $L_2$ together with some additional axioms.} of $L_1$, the relationship between the logics $\ILL, \ILZ, \ILb, \IL_\bot, \CLb$ and $\CLLb$ can be visualized as follows:
\begin{diagram}[tight,width=4em,height=3em]
	\CLLb = \ILZ + \DNE & & \rTo & &  \CLb = \ILZ + \PRO + \DNE \\
	\uTo            &     &              &  & \uTo\\
	\ILZ &  & \rTo   &  &  \IL_\bot = \ILZ + \PRO \\
	\uTo            &     &              &  & \uTo\\
	\ILL &  & \rTo   &  &  \ILb = \ILL + \PRO \\
\end{diagram}
%
%
We will now focus on the various proof translations between these systems. 

\subsection{Modular translations}

For the rest of this section, let $T(\cdot)$ be an $\ILL$ or $\ILZ$ meta-level formula construct, e.g. 
\begin{itemize}
    \item $T(A) = \neg \neg A$ \quad (double negation)
    \item $T(A) = \; \bang A$ \quad (exponentiation)
    \item $T(A) = A = \Id(A)$ \quad (identity)
\end{itemize}

Most formula/proof translations are \emph{modular} in the sense that they are defined by induction on the structure of a given formula/proof. For instance, the Kuroda negative translation from $\CLb$ to $\IL_\bot$, which we will discuss in detail in Section 3, has the following (modular) inductive definition
\[
\begin{array}{rclrcl}
	\Tr{(A \cwedge B)}{\Kuroda} 
	& \pdefin & \Tr{A}{\Kuroda} \cwedge \Tr{B}{\Kuroda} &
	\Tr{(P)}{\Kuroda} 
	& \pdefin & P, 
	\textup{~for $P$ atomic} \\[1mm]
	\Tr{(A \awedge B)}{\Kuroda} 
	& \pdefin & \ \Tr{A}{\Kuroda} \awedge  \Tr{B}{\Kuroda} \quad & 
	\Tr{(\forall x A)}{\Kuroda}
	& \pdefin & \forall x \neg \neg \Tr{A}{\Kuroda} \\[1mm]
	\Tr{(A \oplus B)}{\Kuroda}
	& \pdefin & \Tr{A}{\Kuroda} \oplus \Tr{B}{\Kuroda} &
	\Tr{(\exists x A)}{\Kuroda}
	& \pdefin & \exists x \Tr{A}{\Kuroda} \\[1mm]
	\Tr{(A \lto B)}{\Kuroda} 
	& \pdefin & \Tr{A}{\Kuroda} \lto  \Tr{B}{\Kuroda} &
	\Tr{(\bang A)}{\Kuroda} 
	& \pdefin & \bang  \Tr{A}{\Kuroda}
\end{array}
\]
and $\Trans{A}{\Kuroda} \pdefin \neg \neg \Tr{A}{\Kuroda}$, while the G\"odel-Gentzen negative translation is defined as:
\[
\begin{array}{rclrcl}
	\Trans{(A \cwedge B)}{\Goedel} 
	& \pdefin & \Trans{A}{\Goedel} \cwedge \Trans{B}{\Goedel} &
	\Trans{(P)}{\Goedel} 
	& \pdefin & \neg \neg P, 
	\textup{~for $P$ atomic} \\[1mm]
	\Trans{(A \awedge B)}{\Goedel} 
	& \pdefin & \Trans{A}{\Goedel} \awedge \Trans{B}{\Goedel} & 
	\Trans{(\forall x A)}{\Goedel}
	& \pdefin & \forall x \Trans{A}{\Goedel} \\[1mm]
	\Trans{(A \oplus B)}{\Goedel}
	& \pdefin & \neg \neg (\Trans{A}{\Goedel} \oplus \Trans{B}{\Goedel}) \quad &
	\Trans{(\exists x A)}{\Goedel}
	& \pdefin & \neg \neg \exists x \Trans{A}{\Goedel} \\[1mm]
	\Trans{(A \lto B)}{\Goedel} 
	& \pdefin & \Trans{A}{\Goedel} \lto \Trans{B}{\Goedel} &
	\Trans{(\bang A)}{\Goedel} 
	& \pdefin & \bang \Trans{A}{\Goedel}.
\end{array}
\]
What we observe is that in the first case the double negation is placed ``inside'' the quantifier $\forall$, while in the second case we are placing double negations ``outside'' some of the connectives and quantifiers (namely $\oplus$, $\exists$ and atomic formulas). We will call a transformation such as
%
%
\[ \forall x A \quad \mapsto \quad \forall x \neg \neg A  \]
an \emph{inner transform}, as it is modifying the inner structure of the given formula $\forall x A$, while a transformation of the kind
%
%
\[ \exists x A \quad \mapsto \quad \neg \neg \exists x A \]
we will call an \emph{outer transform}, as it modifies the outer structure of $\exists x A $. 

\begin{defn}[Inner/Outer transforms] \label{def-inner-outer-trans} For any quantifier or modality $Q \in \{ \forall x , \exists x , \bang \}$ and any formula construct $T(\cdot)$, we say that a formula construct $\Transform_Q(A)$ is a \emph{$T$-based $Q$-inner-transform} if
	\begin{equation} \label{inner-transform-def}
		\Transform_Q(A) = Q (T'(A))
	\end{equation}
where $T' \in \{ T, \Id \}$. $\Transform_Q(A)$ is a \emph{$T$-based $Q$-outer-transform} if 
\begin{equation} \label{outer-transform-def}
	\Transform_Q(A) = T'(Q(A))
\end{equation}
where $T' \in \{ T, \Id \}$. In the case where $T$ is not the identity, we say that a transform $\Transform_Q'$ is simpler than another transform $\Transform_Q$ if $\Transform_Q(A)$ is $Q(T(A))$ or $T(Q(A))$ while $\Transform_Q'(A) = Q(A)$ (outer or inner $T$ omitted). \\[1mm]
Similarly, for any connective $\Box \in \{ \cwedge, \awedge, \oplus, \lto \}$, we say that $\Transform_\Box(A, B)$ is a \emph{$T$-based $\Box$-inner-transform} if
\[ \Transform_\Box(A, B) = T'(A) \squareOp T''(B), \]
where $T', T'' \in \{ T, \Id \}$. $\Transform_\Box(A, B)$ is a \emph{$T$-based $\Box$-outer-transform} if
\[
\Transform_\Box(A, B) = T'(A \squareOp B)
\] 
where $T' \in \{ T, \Id \}$. As above, in the case for which $T$ is not the identity,  we say that a transform $\Transform_\Box'$ is simpler than another transform $\Transform_\Box$ if $\Transform_\Box$ uses $T$'s in all places that $\Transform_\Box'$ does, but not conversely. For instance, $A \squareOp B, T(A) \squareOp B$ and $A \squareOp T(B)$ are all simpler than $T(A) \squareOp T(B)$, but $A \squareOp T(B)$ is not simpler than $T(A) \squareOp B$.
\end{defn}

In order to illustrate the notion of an inner or outer transform, consider $T(A) = \neg \neg A$ and the connective $A \lto B$. The transform 
\[ \Transform_{\lto}(A, B) = A \lto \neg \neg B \]
is a $\neg \neg$-based $\lto$-inner transform, whereas 
\[ \Transform_{\lto}(A, B) = \neg \neg (A \lto B) \] 
is a $\neg \neg$-based $\lto$-outer transform. 

For each formula construct $T$ and $\Box \in \{ \cwedge, \awedge, \oplus, \lto \}$ or $Q \in \{ \forall x , \exists x , \bang \}$, there is a fixed finite number of $T$-based inner- or outer-transforms. For instance, there are only two $T$-based $\Box$-outer transforms: $T(A \squareOp B)$ and $A \squareOp B$, and there are four $T$-based $\Box$-inner transforms: $T(A) \squareOp T(B)$, $T(A) \squareOp B$, $A \squareOp T(B)$, and $A \squareOp B$.

Note how the Kuroda translation $\Tr{A}{\Kuroda}$ is determined by $\neg \neg$-based inner-transforms, e.g. 
\[ 
\Transform_{\forall x}^{\Kuroda}(A) = \forall x \neg \neg A.\] 
%
We define the notion of a general modular inner $T$-translation as follows:

\begin{defn}[Modular inner $T$-translation] We say that a formula translation $\genericTrans{(\cdot)}$, over the language of $\ILL$ or $\ILZ$, is a \emph{modular inner $T$-translation} if, for some $T$-based inner-transforms $\ITTr{\forall x}, \ITTr{\exists x}, \ITTr{!}, \ITTr{\cwedge}, \ITTr{\awedge}, \ITTr{\oplus}, \ITTr{\lto}$ we have that
	\[ \genericTrans{A} \equiv T(\genericTr{A}) \]
	where $\genericTr{(\cdot)}$ is defined inductively as:
	\[
	\begin{array}{rclrcl}
		\genericTr{(A \cwedge B)} & \pdefin & \ITTr{\cwedge}(\genericTr{A}, \genericTr{B}) &
		\genericTr{P} & \pdefin & P, \textup{~for $P$ atomic} \\[1mm]
		\genericTr{(A \awedge B)} & \pdefin & \ITTr{\awedge}(\genericTr{A}, \genericTr{B}) &
		\genericTr{(\forall x A)} & \pdefin & \ITTr{\forall x} (\genericTr{A}) \\[1mm]
		\genericTr{(A \oplus B)} & \pdefin & \ITTr{\oplus}(\genericTr{A}, \genericTr{B}) &
		\genericTr{(\exists x A)} & \pdefin & \ITTr{\exists x} (\genericTr{A}) \\[1mm]
		\genericTr{(A \lto B)} & \pdefin & \ITTr{\lto}(\genericTr{A}, \genericTr{B}) \quad &
		\genericTr{(! A)} & \pdefin & \ITTr{!}(\genericTr{A}).
	\end{array}
	\]
\end{defn}

Similary, the G\"odel-Gentzen $\neg \neg$-translation $\Tr{A}{\Goedel}$ arises from a set of $\neg \neg$-based outer-transforms, e.g. 
\[ 
\Transform_{\exists x}^{\Goedel}(A) = \neg \neg \exists x A
\quad
\mbox{and} 
\quad 
\Transform_{\oplus}^{\Goedel}(A, B) = \neg \neg (A \oplus B). \]
Hence, we also define the notion of a general modular outer $T$-translation as follows: 

\begin{defn}[Modular outer $T$-translation] We say that a formula translation $\genericTrans{(\cdot)}$, over the language of $\ILL$ or $\ILZ$, is a \emph{modular outer $T$-translation} if it is defined inductively as
	\[
	\begin{array}{rclrcl}
		\genericTrans{(A \cwedge B)} & \pdefin & \ITTr{\cwedge}(\genericTrans{A}, \genericTrans{B}) &
		\genericTrans{P} & \pdefin & T(P), \textup{~for $P$ atomic}\\[1mm]
		\genericTrans{(A \awedge B)} & \pdefin & \ITTr{\awedge}(\genericTrans{A}, \genericTrans{B}) &
		\genericTrans{(\forall x A)} & \pdefin & \ITTr{\forall x} (\genericTrans{A})\\[1mm]
		\genericTrans{(A \oplus B)} & \pdefin & \ITTr{\oplus}(\genericTrans{A}, \genericTrans{B}) &
		\genericTrans{(\exists x A)} & \pdefin & \ITTr{\exists x} (\genericTrans{A})\\[1mm]
		\genericTrans{(A \lto B)} & \pdefin & \ITTr{\lto}(\genericTrans{A}, \genericTrans{B}) \quad
		\quad &
		\genericTrans{(!A)} & \pdefin & \ITTr{!}(\genericTrans{A})
	\end{array}
	\]
	where $\ITTr{\forall x}, \ITTr{\exists x}, \ITTr{!}, \ITTr{\cwedge}, \ITTr{\awedge}, \ITTr{\oplus}, \ITTr{\lto}$ are $T$-based outer-transforms.
\end{defn}

\begin{remark} As we will see, some formula translations can be presented both as modular inner $T$-translations or modular outer $T$-translations. So, one should think of the formulation of the translation via the inner and outer transforms as possible presentations (or computations) of the formula translation. Notable cases are the Kolmogorov translation (see Def. \ref{def:Kolmogorov-translations}) and the full Girard translation (see Def. \ref{def-Girard-full}).
\end{remark}

Obviously, we want translations from a logic $S_1$ into another logic $S_2$ that are ``sound'', in the sense that they preserve provability in a non-trivial way:  

\begin{defn}[Sound translations] 
	Given two $\ILL$ theories $S_1$ and $S_2$. We say that $\genericTrans{(\cdot)}$ is a sound translation from $S_1$ to $S_2$ if for all formulas $A$ we have that
	\begin{itemize}
		\item[(i)] If $\vdash_{S_1} A$ then $\vdash_{S_2} \genericTrans{A}$, and
		\item[(ii)] $A$ and $\genericTrans{A}$ are equivalent over $S_1$.
	\end{itemize}
\end{defn}

All the translations we are going to present in this paper are sound modular $T$-translations for some extensions of $\ILL$ and some $T(\cdot)$. 

\begin{defn}[Equivalent translations] 
	Two sound translations $\genericTrans{(\cdot)}$ and $\genericTransPrime{(\cdot)}$ from $S_1$ to $S_2$ are \emph{equivalent} if $\genericTrans{A}$ and $\genericTransPrime{A}$ are equivalent over $S_2$, for all $S_1$ formulas $A$.
\end{defn}

\subsection{Simplifying translations (from inside and outside)}

Although two different translations might be equivalent, it can still be that one is ``simpler'' than the other. The following definition formalises this notion.

\begin{defn}[Simplified translation] Let $\genericTrans{(\cdot)}$ and 
	$\genericTransPrime{(\cdot)}$ be two equivalent translations. We say that $\genericTransPrime{(\cdot)}$ is a \emph{simplification of} $\genericTrans{(\cdot)}$ if the transforms of $\genericTransPrime{(\cdot)}$ are either identical to those of $\genericTrans{(\cdot)}$ or simpler -- in the sense of Definition \ref{def-inner-outer-trans} -- than the corresponding transforms of $\genericTrans{(\cdot)}$ (and must be simpler in at least one case). 
\end{defn}



Our motivation for the present work is twofold. As we already mentioned above, expressing the logical systems considered in this paper as extensions of Intuitionistic Linear Logic ($\ILL$) allows us to identify the specific axioms that distinguish them, enabling the construction of straightforward canonical modular translations between these systems. Secondly, we aim to show that these canonical formula translations (being modular) can be systematically transformed into simpler (also modular) ones. We refer to this process, applied to a modular translation, as a ``simplification''. Through this approach, as we will see, many of the well-known translations in the literature can be obtained, thus offering an explanation for how seemingly non-trivial modular translations arise. The ``simplification'' of a modular translation can happen in two ways: \emph{from outside} or \emph{from inside}.


\begin{defn}[Simplification from outside] \label{def-simp-from-outside} Given two equivalent modular outer $T$-translations $\genericTrans{(\cdot)}$ and $\genericTransPrime{(\cdot)}$, we say that  $\genericTransPrime{(\cdot)}$ is a \emph{simplification from outside of} $\genericTrans{(\cdot)}$  if $\genericTransPrime{(\cdot)}$ is a \emph{simplification of} $\genericTrans{(\cdot)}$, and, when we consider the set of formula reductions
	\[
	\begin{array}{rcl}
		\IT{{\rm Tr}}{\Box}(T(A), T(B)) 
		& \mapsto & \IT{{\rm Tr}'}{\Box}(T(A), T(B)) \\[1mm]
		\IT{{\rm Tr}}{Q}(T(A)) 
		& \mapsto & \IT{{\rm Tr}'}{Q}(T(A))
	\end{array}
	\]
	for $\Box \in \{ \oplus, \cwedge, \awedge, \lto \}$ and $Q \in \{ \forall x, \exists x, \bang \}$, we have that $\genericTransPrime{A}$ can be obtained from $\genericTrans{A}$ by applying these reductions recursively, starting from whole formula and working towards the atomic formulas.
\end{defn}

Therefore, if $\genericTransPrime{(\cdot)}$ is a simplification from outside of $\genericTrans{(\cdot)}$, the $T$-based inner-transforms of $\genericTrans{(\cdot)}$ and $\genericTransPrime{(\cdot)}$ will tell us how to systematically simplify $\genericTrans{(\cdot)}$ to obtain  $\genericTransPrime{(\cdot)}$. For instance, we will see in the next section that the G\"odel-Gentzen negative translation is a simplification from outside of (an outer-presentation of) the Kolmogorov translation, whereas the Kuroda translation will be shown to be a simplification from inside of (an inner-presentation of) the Kolmogorov translation.

\begin{defn}[Simplification from inside] \label{def-simp-from-inside} Given two equivalent modular inner $T$-translations $\genericTrans{(\cdot)}$ and $\genericTransPrime{(\cdot)}$, we say that $\genericTransPrime{(\cdot)}$ is a \emph{simplification from inside of} $\genericTrans{(\cdot)}$ if $\genericTransPrime{(\cdot)}$ is a \emph{simplification of} $\genericTrans{(\cdot)}$, and, when we consider the set of formula reductions
	\[
	\begin{array}{rcl}
		T(\IT{{\rm Tr}}{\Box}(A, B)) 
		& \mapsto & T(\IT{{\rm Tr}'}{\Box}(A, B)) \\[1mm]
		T(\IT{{\rm Tr}}{Q}(A)) 
		& \mapsto & T(\IT{{\rm Tr}'}{Q}(A))
	\end{array}
	\]
	for $\Box \in \{ \oplus, \cwedge, \awedge, \lto \}$ and $Q \in \{ \forall x, \exists x, \bang \}$, we find that $\genericTransPrime{A}$ can be obtained from $\genericTrans{A}$ by applying these reductions inductively, starting from the atomic formulas and working toward the whole formula.
\end{defn}

\section{Translations from $\CLb$ to $\IL_\bot$}
\label{sec-cl-to-il}

Let us now see how the general definitions of a translation simplification of the previous section apply to the concrete case of \emph{negative translations}, such as the Kolmogorov, G\"odel-Gentzen and Kuroda translations. In the following sections, we will carry out a similar study of the Girard translations of standard logic into linear logic. 

%

\subsection{Kolmogorov, G\"odel-Gentzen and Kuroda translations}


Kolmogorov's negative translation \cite{kol} of $\CL$ to $\IL$ works by placing double negations in front of each subformula inductively. The same can be done for translating $\CLb$ to $\IL_\bot$ -- in this case we are working with $T(A) = \neg \neg A$. In fact, there are two ways to formally present this inductively, which we will call $\Kolm{o}$ and $\Kolm{i}$. 

\begin{defn}[Kolmogorov translation \cite{kol}] \label{def:Kolmogorov-translations} The \emph{outer} presentation of Kolmogorov's translation, which we will denote by $\Trans{(\cdot)}{\Kolm{o}}$, uses of outer transforms 
\begin{itemize}
    \item $\Transform_Q^{\Kolm{o}}(A) = \neg \neg Q A$, for $Q \in \{ \forall x , \exists x , \bang \}$, and 
    \item $\Transform_\Box^{\Kolm{o}}(A, B) = \neg \neg (A \squareOp B)$, for $\Box \in \{ \cwedge, \awedge, \oplus, \lto \}$, 
\end{itemize}
giving rise to a modular outer $\neg \neg$-translation $\Trans{A}{\Kolm{o}}$, which can also be inductively defined as:
\[
\begin{array}{rclrcl}
    \Trans{(A \cwedge B)}{\Kolm{o}} 
    & \pdefin & \neg \neg (\Trans{A}{\Kolm{o}} \cwedge \Trans{B}{\Kolm{o}}) &
    \Trans{P}{\Kolm{o}} 
    & \pdefin & \neg \neg P, 
    \textup{~for $P$ atomic} \\[1mm]
    \Trans{(A \awedge B)}{\Kolm{o}} 
    & \pdefin & \neg \neg (\Trans{A}{\Kolm{o}} \awedge \Trans{B}{\Kolm{o}}) & 
    \Trans{(\forall x A)}{\Kolm{o}}
    & \pdefin & \neg \neg \forall x \Trans{A}{\Kolm{o}} \\[1mm]
    \Trans{(A \oplus B)}{\Kolm{o}}
    & \pdefin & \neg \neg (\Trans{A}{\Kolm{o}} \oplus \Trans{B}{\Kolm{o}}) &
    \Trans{(\exists x A)}{\Kolm{o}}
    & \pdefin & \neg \neg \exists x \Trans{A}{\Kolm{o}} \\[1mm]
    \Trans{(A \lto B)}{\Kolm{o}} 
    & \pdefin & \neg \neg (\Trans{A}{\Kolm{o}} \lto \Trans{B}{\Kolm{o}}) \quad &
    \Trans{(\bang A)}{\Kolm{o}} 
    & \pdefin & \neg \neg \bang \Trans{A}{\Kolm{o}}.
\end{array}
\]
Alternatively, we can also make use of inner transforms
\begin{itemize}
    \item $\Transform_Q^{\Kolm{i}}(A) = Q \neg \neg A$, for $Q \in \{ \forall x , \exists x , \bang \}$, and 
    \item $\Transform_\Box^{\Kolm{i}}(A, B) = \neg \neg A \squareOp \neg \neg B$, for $\Box \in \{ \cwedge, \awedge, \oplus, \lto \}$,
\end{itemize}
giving rise to a modular inner $\neg \neg$-translation
\eqleft{\Trans{A}{\Kolm{i}} \pdefin \neg \neg \Tr{A}{\Kolm{i}}}
where $\Tr{A}{\Kolm{i}}$ is defined inductively as:
\[
\begin{array}{rclrcl}
    \Tr{(A \cwedge B)}{\Kolm{i}} 
    & \pdefin & \neg \neg \Tr{A}{\Kolm{i}} \cwedge \neg \neg \Tr{B}{\Kolm{i}} &
    \Tr{P}{\Kolm{i}} 
    & \pdefin & P, 
    \textup{~for $P$ atomic} \\[1mm]
    \Tr{(A \awedge B)}{\Kolm{i}} 
    & \pdefin & \neg \neg \Tr{A}{\Kolm{i}} \awedge \neg \neg \Tr{B}{\Kolm{i}} & 
    \Tr{(\forall x A)}{\Kolm{i}}
    & \pdefin & \forall x \neg \neg \Tr{A}{\Kolm{i}} \\[1mm]
    \Tr{(A \oplus B)}{\Kolm{i}}
    & \pdefin & \neg \neg \Tr{A}{\Kolm{i}} \oplus \neg \neg \Tr{B}{\Kolm{i}} &
    \Tr{(\exists x A)}{\Kolm{i}}
    & \pdefin & \exists x \neg \neg \Tr{A}{\Kolm{i}} \\[1mm]
    \Tr{(A \lto B)}{\Kolm{i}} 
    & \pdefin & \neg \neg \Tr{A}{\Kolm{i}} \lto \neg \neg \Tr{B}{\Kolm{i}} \quad \quad &
    \Tr{(\bang A)}{\Kolm{i}} 
    & \pdefin & \bang \neg \neg \Tr{A}{\Kolm{i}}.
\end{array}
\]
\end{defn}

It is easy to check, by induction on the structure of $A$, that $\Trans{A}{\Kolm{o}} \equiv \Trans{A}{\Kolm{i}}$, for all formulas $A$. Intuitively, we can think of $\Trans{(\cdot)}{\Kolm{o}}$ as working bottom-up: start by placing double negations in front of all atomic formulas, and then work your way up the syntax tree placing double negations in front of all connectives, quantifiers and modality. $\Trans{(\cdot)}{\Kolm{i}}$, on the other hand, produces the same result, but we are working top-down: placing a double negation in front of the whole formula, and then recursively going down the syntax tree placing double negations in front of all subformulas.

The G\"odel-Gentzen translation, on the other hand, only has a modular outer $\neg \neg$-translation presentation:

\begin{defn}[Gödel-Gentzen translation \cite{god}] \label{def-god} The G\"odel-Gentzen negative translation is a modular outer $\neg \neg$-translation defined by $\neg \neg$-based outer transforms:
	\[
	\begin{array}{cclccl}
		\IT{\Goedel}{\cwedge}(A, B) 
		& \pdefin & A \cwedge B &  &  & \\[1mm]
		\IT{\Goedel}{\awedge}(A, B) 
		& \pdefin & A \awedge B & 
		\IT{\Goedel}{\forall x} (A)
		& \pdefin & \forall x A \\[1mm]
		\IT{\Goedel}{\oplus}(A, B)
		& \pdefin & \neg \neg (A \oplus B) \quad &
		\IT{\Goedel}{\exists x} (A)
		& \pdefin & \neg \neg \exists x A \\[1mm]
		\IT{\Goedel}{\lto}(A, B) 
		& \pdefin & A \lto B \quad \quad &
		\IT{\Goedel}{!}(A) 
		& \pdefin &  \bang A.
	\end{array}
	\]
	Hence, $\Trans{A}{\Goedel}$ is defined inductively as:
	\[
	\begin{array}{rclrcl}
		\Trans{(A \cwedge B)}{\Goedel} 
		& \pdefin & \Trans{A}{\Goedel} \cwedge \Trans{B}{\Goedel} &
		\Trans{P}{\Goedel} 
		& \pdefin & \neg \neg P, 
		\textup{~for $P$ atomic} \\[1mm]
		\Trans{(A \awedge B)}{\Goedel} 
		& \pdefin & \Trans{A}{\Goedel} \awedge \Trans{B}{\Goedel} & 
		\Trans{(\forall x A)}{\Goedel}
		& \pdefin & \forall x \Trans{A}{\Goedel} \\[1mm]
		\Trans{(A \oplus B)}{\Goedel}
		& \pdefin & \neg \neg (\Trans{A}{\Goedel} \oplus \Trans{B}{\Goedel})  \quad &
		\Trans{(\exists x A)}{\Goedel}
		& \pdefin & \neg \neg \exists x \Trans{A}{\Goedel} \\[1mm]
		\Trans{(A \lto B)}{\Goedel} 
		& \pdefin & \Trans{A}{\Goedel} \lto \Trans{B}{\Goedel} &
		\Trans{(\bang A)}{\Goedel} 
		& \pdefin &  \bang \Trans{A}{\Goedel}.
	\end{array}
	\]
	
\end{defn}

The Kuroda negative translation, on the other hand, is a modular inner $\neg \neg$-translation:

\begin{defn}[Kuroda translation \cite{kur}] \label{def-kur} The Kuroda negation translation is a modular inner $\neg \neg$-translation defined by $\neg \neg$-based inner transforms:
	\[
	\begin{array}{rclrcl}
		\IT{\Kuroda}{\cwedge}(A, B) 
		& \pdefin & A \cwedge B &  &  & \\[1mm]
		\IT{\Kuroda}{\awedge}(A, B) 
		& \pdefin & A \awedge B \quad & 
		\IT{\Kuroda}{\forall x} (A)
		& \pdefin & \forall x \neg \neg A \\[1mm]
		\IT{\Kuroda}{\oplus}(A, B)
		& \pdefin & A \oplus B &
		\IT{\Kuroda}{\exists x} (A)
		& \pdefin & \exists x A \\[1mm]
		\IT{\Kuroda}{\lto}(A, B) 
		& \pdefin & A \lto B &
		\IT{\Kuroda}{!}(A) 
		& \pdefin & \bang A.
	\end{array}
	\]
	Hence,
	\eqleft{\Trans{A}{\Kuroda} \pdefin \neg \neg \Tr{A}{\Kuroda}}
	where $\Tr{A}{\Kuroda}$ is defined inductively as:
	\[
	\begin{array}{rclrcl}
		\Tr{(A \cwedge B)}{\Kuroda} 
		& \pdefin & \Tr{A}{\Kuroda} \cwedge \Tr{B}{\Kuroda} &
		\Tr{(P)}{\Kuroda} 
		& \pdefin & P, 
		\textup{~for $P$ atomic} \\[1mm]
		\Tr{(A \awedge B)}{\Kuroda} 
		& \pdefin &  \Tr{A}{\Kuroda} \awedge \Tr{B}{\Kuroda} \quad & 
		\Tr{(\forall x A)}{\Kuroda}
		& \pdefin & \forall x \neg \neg \Tr{A}{\Kuroda} \\[1mm]
		\Tr{(A \oplus B)}{\Kuroda}
		& \pdefin & \Tr{A}{\Kuroda} \oplus \Tr{B}{\Kuroda} &
		\Tr{(\exists x A)}{\Kuroda}
		& \pdefin & \exists x \Tr{A}{\Kuroda} \\[1mm]
		\Tr{(A \lto B)}{\Kuroda} 
		& \pdefin & \Tr{A}{\Kuroda} \lto \Tr{B}{\Kuroda} &
		\Tr{(\bang A)}{\Kuroda} 
		& \pdefin & \bang \Tr{A}{\Kuroda}.
	\end{array}
	\]
\end{defn}


\begin{prop} \label{prop-sound-double-negation-trans} The above four modular $\neg \neg$-translations (inner Kolmogorov, outer Kolmogorov, G\"odel-Gentzen and Kuroda) are all sound translations of $\CLb$ to $\IL_\bot$, and equivalent to each other.
\end{prop}

\begin{proof} Consider, for example, the inner Kolmogorov translation of $\CLb$ into $\IL_\bot$. Clearly, $\CLb$ proves the equivalence between $A$ and $\Trans{A}{\Kolm{i}}$. Now, we can show that $\Trans{A}{\Kolm{i}}$ is derivable in $\IL_\bot$ whenever $A$ is derivable in $\CLb$, by induction on the $\CLb$-derivation of $A$. The main challenge is to prove that the translations of the axioms of $\CLb$ are derivable in $\IL_\bot$ (it is easy to check that the translations respect logical rules). But, except for $\DNE$, the translation of each axiom of $\CLb$ is easily derivable in $\IL_\bot$. Finally, one can also show that $\IL_\bot$ proves $\Trans{B}{\Kolm{i}}$, for any instance $B$ of $\DNE$. The equivalence between the translations (over $\IL_\bot$) can be shown using the $\IL_\bot$-equivalences of Proposition \ref{not-not-equivalences}.
\end{proof}

Although all four translations above are equivalent, one can observe that the G\"odel-Gentzen translation follows the same pattern as the outer Kolmogorov $\neg \neg$-translation, but has fewer double negation in certain places (namely, $\otimes,\awedge, \lto$,$\forall$ and $\bang$),
%
%
whereas the Kuroda translation is similar to the inner Kolmorov $\neg \neg$-translation, but again with fewer double negations in certain places (namely $\cwedge$, $\awedge$, $\oplus$, $\lto$, $\exists $ and $\bang$).
%
%

\newcommand{\ML}{ML}

\subsection{$\Trans{(\cdot)}{\Goedel}$ is a simplification from outside of $\Trans{(\cdot)}{\Kolm{o}}$}

One might ask, what is it that allows us to simplify the outer Kolmogorov translation $\Trans{(\cdot)}{\Kolm{o}}$ into the G\"odel-Gentzen translation $\Trans{(\cdot)}{\Goedel}$, omitting double negations in additive and multiplicative conjunctions ($\awedge$, $\otimes$), implications ($\lto$), universal quantifiers ($\forall$) and $(\bang)$? Similarly, why can we simplify the inner Kolmogorov translation $\Trans{(\cdot)}{\Kolm{i}}$ into the Kuroda translation $\Trans{(\cdot)}{\Kuroda}$, omitting double negations in the logical connectives, in the existential quantifiers ($\exists$) and in the modalities $(\bang)$? We observe the following:

\begin{prop}[$\Trans{(\cdot)}{\Goedel}$ is a simplification from outside of $\Trans{(\cdot)}{\Kolm{o}}$] \label{thm-goedel-simplifies-kol} For any formula $A$, one can obtain $\Trans{A}{\Goedel}$ from $\Trans{A}{\Kolm{o}}$ by systematically applying the following formula reductions starting from the whole formula and inductively applying these to the subformulas:
	\[
	\begin{array}{rcl}
		\underbrace{\neg \neg (\neg \neg A \awedge \neg \neg B)}_{\IT{\Kolm{o}}{\awedge}(\neg \neg A, \neg \neg B)} 
		& \quad \mapsto \quad & 
		\underbrace{\neg \neg A \awedge \neg \neg B}_{\IT{\Goedel}{\awedge}(\neg \neg A, \neg \neg B)} \\[7mm]
		\underbrace{\neg \neg (\neg \neg A \cwedge \neg \neg B)}_{\IT{\Kolm{o}}{\cwedge}(\neg \neg A, \neg \neg B)} 
		& \quad \mapsto \quad & 
		\underbrace{\neg \neg A \cwedge \neg \neg B}_{\IT{\Goedel}{\cwedge}(\neg \neg A, \neg \neg B)} \\[7mm]
		\underbrace{\neg \neg (\neg \neg A \lto \neg \neg B)}_{\IT{\Kolm{o}}{\lto}(\neg \neg A, \neg \neg B)} 
		& \quad \mapsto \quad & 
		\underbrace{\neg \neg A \lto \neg \neg B}_{\IT{\Goedel}{\lto}(\neg \neg A, \neg \neg B)} \\[7mm]
		\underbrace{\neg \neg \forall x \neg \neg A}_{\IT{\Kolm{o}}{\forall x}(\neg \neg A)}
		& \quad \mapsto \quad &
		\underbrace{\forall x \neg \neg A}_{\IT{\Goedel}{\forall x}(\neg \neg A)}
        \\[7mm]
		\underbrace{\neg \neg \bang \neg \neg A}_{\IT{\Kolm{o}}{!}(\neg \neg A)}
		& \quad \mapsto \quad &
		\underbrace{\bang \neg \neg  A}_{\IT{\Goedel}{!}(\neg \neg A)}.
	\end{array}
	\]
	Hence, $\Trans{(\cdot)}{\Goedel}$ is a simplification from outside of $\Trans{(\cdot)}{\Kolm{o}}$ -- in the sense of Definition \ref{def-simp-from-outside}. Moreover, by Proposition \ref{not-not-equivalences}, the reductions above are reversible in $\IL_\bot$, so $\Trans{A}{\Goedel}$ and $\Trans{A}{\Kolm{o}}$ are $\IL_\bot$-equivalent.
\end{prop}
\begin{proof} Let us write $A \mapsto^o B$ when $B$ is obtained from $A$ by systematically applying the reductions $\mapsto$ ``from the outside''. We are going to prove by induction on the logical structure of $A$ that $A^{\Kolm{o}} \mapsto^o \Trans{A}{\Goedel}$. We will make use of the fact that for any formula $A$ there is a formula $A'$ such that $\Trans{A}{\Kolm{o}} = \neg \neg A'$. \\[1mm]
If $A = P$ then $\Trans{A}{\Kolm{o}} = \neg \neg P = \Trans{A}{\Goedel}$, and no reduction needs to be applied. \\[1mm]
If $A = B \oplus C$ then $\Trans{A}{\Kolm{o}} = \neg \neg (\Trans{B}{\Kolm{o}} \oplus \Trans{C}{\Kolm{o}})$. In this case, no reduction is applicable to the whole formula. By induction hypothesis we have that $\Trans{B}{\Kolm{o}} \mapsto^o \Trans{B}{\Goedel}$ and $\Trans{C}{\Kolm{o}} \mapsto^o \Trans{C}{\Goedel}$. Hence,
\[ 
    \Trans{A}{\Kolm{o}} = 
    \neg \neg (\Trans{B}{\Kolm{o}} \oplus \Trans{C}{\Kolm{o}}) \mapsto^o 
    \neg \neg (\Trans{B}{\Goedel} \oplus \Trans{C}{\Goedel}) = 
    (B \oplus C)^{\Goedel} = \Trans{A}{\Goedel}. 
\]
If $A = B \squareOp C$, with $ \squareOp \in \{\cwedge, \awedge, \lto \}$ then 
\[ \Trans{A}{\Kolm{o}} = \neg \neg (\Trans{B}{\Kolm{o}} \squareOp \Trans{C}{\Kolm{o}}) = \neg \neg (\neg \neg B' \squareOp \neg \neg C') \]
for some formulas $B'$ and $C'$. Hence, we can apply one of the reductions to the whole formula so that
\[ 
\Trans{A}{\Kolm{o}} = \neg \neg (\neg \neg B' \squareOp \neg \neg C') \mapsto \neg \neg B' \squareOp \neg \neg C' = \Trans{B}{\Kolm{o}} \squareOp \Trans{C}{\Kolm{o}}.
\]
By induction hypothesis we have that $\Trans{B}{\Kolm{o}} \mapsto^o \Trans{B}{\Goedel}$ and $\Trans{C}{\Kolm{o}} \mapsto^o \Trans{C}{\Goedel}$. Hence,
\[ \Trans{A}{\Kolm{o}} \mapsto \Trans{B}{\Kolm{o}} \squareOp \Trans{C}{\Kolm{o}} \mapsto^o \Trans{B}{\Goedel} \squareOp \Trans{C}{\Goedel} = (B \squareOp C)^{\Goedel} = \Trans{A}{\Goedel}. \]
If $A = \exists x B$ then $\Trans{A}{\Kolm{o}} = \neg \neg \exists x \Trans{B}{\Kolm{o}}$. In this case, no reduction is applicable to the whole formula. By induction hypothesis, $\Trans{B}{\Kolm{o}} \mapsto^o \Trans{B}{\Goedel}$. Hence
\[ \Trans{A}{\Kolm{o}} = \neg \neg \exists x \Trans{B}{\Kolm{o}} \mapsto^o \neg \neg \exists x \Trans{B}{\Goedel} = (\exists x B)^{\Goedel} = \Trans{A}{\Goedel}. \]
If $A = Q B$, with $Q \in \{\forall x, \bang\}$,  then $\Trans{A}{\Kolm{o}} = \Trans{(Q B)}{\Kolm{o}} = \neg \neg Q \Trans{B}{\Kolm{o}} = \neg \neg Q \neg \neg B'$, for some formula $B'$. Hence, we can apply one of the reductions to the whole formula so that
\[
\Trans{A}{\Kolm{o}} = \neg \neg Q \neg \neg B' \mapsto Q \neg \neg B' = Q  \Trans{B}{\Kolm{o}} \mapsto^o Q \Trans{B}{\Goedel} = \Trans{A}{\Goedel}. \]
That concludes the proof.
\end{proof}

\begin{obs} Note that the order of in which the reductions are applied is important. Consider 
	\begin{equation} \label{ex3-formula}
        \neg \neg (\neg \neg(\neg \neg P\awedge \neg \neg Q) \cwedge \neg \neg R).
    \end{equation}
	If we applied reductions to the whole formula (from inside) we would get the following (and the process stops):
	\begin{equation} \label{ex3-other}
        \neg \neg( (\neg \neg P\awedge \neg \neg Q)\cwedge \neg \neg R). 
    \end{equation}
	But if we applied the reductions to $(\ref{ex3-formula})$ from  outside we get
	\begin{equation}
        \neg \neg (\neg \neg P\awedge \neg \neg Q) \cwedge \neg \neg R 
    \end{equation}
	and then
	\begin{equation} \label{ex3-gg-trans}
        (\neg \neg P\awedge \neg \neg Q) \cwedge \neg \neg R. 
    \end{equation}
	Although (\ref{ex3-other}) and (\ref{ex3-gg-trans}) are equivalent in $\IL_\bot$, (\ref{ex3-gg-trans}) is certainly simpler than (\ref{ex3-other}). The formula (\ref{ex3-formula}) is the $\Kolm{o}$-translation of $(P \awedge Q) \cwedge R$, and (\ref{ex3-gg-trans}) is the G\"odel-Gentzen  translation of $(P \awedge Q) \cwedge R$ obtained by applying the reductions from the outside.
\end{obs}

\begin{obs}[Maximality of simplifications] As in \cite{ol,ol2}, we could also argue that $\Trans{(\cdot)}{\Goedel}$ is not just a simplification from outside of $\Trans{(\cdot)}{\Kolm{o}}$, but it is actually a \emph{maximal simplification}, in the sense that simplifying the clauses for additive disjunction and existential quantifiers would not be possible (see Proposition \ref{not-not-equivalences}). In order to keep this paper as concise as possible we will not discuss the issue of maximality of the simplifications any further, but the reader will be able to verify, using Propositions \ref{not-not-equivalences} and \ref{bang-equivalences}, that in all cases, the set of simplifications we are using is indeed maximal.
\end{obs}

\subsection{$\Trans{(\cdot)}{\Kuroda}$ is a simplification from inside of $\Trans{(\cdot)}{\Kolm{i}}$}

A similar phenomena holds for the inner presentation of the Kolmogorov translation $\Trans{(\cdot)}{\Kolm{i}}$ and Kuroda's translation $\Trans{(\cdot)}{\Kuroda}$.

\begin{prop}[$\Trans{(\cdot)}{\Kuroda}$ is a simplification from inside of $\Trans{(\cdot)}{\Kolm{i}}$] \label{prop-kur-simplifies-kolm} For any formula $A$, one can obtain $\Trans{A}{\Kuroda}$ from $\Trans{A}{\Kolm{i}}$ by systematically applying the following formula reductions starting from the atomic formulas and inductively applying these to composite formulas:
	\[
	\begin{array}{rcl}
		\underbrace{\neg \neg (\neg \neg A \cwedge \neg \neg B)}_{\neg \neg \IT{\Kolm{i}}{\cwedge}(A, B)} 
		& \quad \mapsto \quad & 
		\underbrace{\neg \neg (A \cwedge B)}_{\neg \neg \IT{\Kuroda}{\cwedge}(A, B)} \\[3mm]
		    \underbrace{\neg \neg (\neg \neg A \awedge \neg \neg B)}_{\neg \neg \IT{\Kolm{i}}{\awedge}(A, B)} 
		       & \quad \mapsto \quad & 
		      \underbrace{\neg \neg (A \awedge B)}_{\neg \neg \IT{\Kuroda}{\awedge}(A, B)} \\[1mm]
		\underbrace{\neg \neg (\neg \neg A \oplus \neg \neg B)}_{\neg \neg \IT{\Kolm{i}}{\oplus}(A, B)} 
		& \quad \mapsto \quad & 
		\underbrace{\neg \neg (A \oplus B)}_{\neg \neg \IT{\Kuroda}{\oplus}(A, B)} \\[3mm]
		\underbrace{\neg \neg (\neg \neg A \lto \neg \neg B)}_{\neg \neg \IT{\Kolm{i}}{\lto}(A, B)} 
		& \quad \mapsto \quad & 
		\underbrace{\neg \neg (A \lto B)}_{\neg \neg \IT{\Kuroda}{\lto}(A, B)} \\[3mm]
		\underbrace{\neg \neg \exists x \neg \neg A}_{\neg \neg \IT{\Kolm{i}}{\exists x}(A)}
		& \quad \mapsto \quad &
		\underbrace{\neg \neg \exists x A}_{\neg \neg \IT{\Kuroda}{\exists x}(A)}
  \\[3mm]
		\underbrace{\neg \neg \bang \neg \neg A}_{\neg \neg \IT{\Kolm{i}}{!}(A)}
		& \quad \mapsto \quad &
		\underbrace{\neg \neg \bang A}_{\neg \neg \IT{\Kuroda}{!}(A)}.
	\end{array}
	\]
    Hence, $\Trans{(\cdot)}{\Kuroda}$ is a simplification from inside of $\Trans{(\cdot)}{\Kolm{i}}$ -- in the sense of Definition \ref{def-simp-from-inside}.
    Moreover, by Proposition \ref{not-not-equivalences}, the reductions above are reversible in $\IL_\bot$, so that $\Trans{A}{\Kuroda}$ and $\Trans{A}{\Kolm{i}}$ are $\IL_\bot$-equivalent.
\end{prop}
\begin{proof} The proof is similar to that of Proposition \ref{thm-goedel-simplifies-kol}, except that we define $A \mapsto^i B$, meaning that $B$ is obtained from $A$ by systematically applying the reductions $\mapsto$ ``from the inside'', and then show by induction on $A$ that $\Trans{A}{\Kolm{i}}  \mapsto^i \Trans{A}{\Kuroda}$. 
\end{proof}


\section{Translations from $\ILb$ to $\ILL$}
\label{sec-il-to-ill}

We have seen that we can translate $\IL$ into an extension of $\ILL$, which we called $\ILb = \ILL + \PRO$ (cf. Definition \ref{dagger-trans} and Proposition \ref{prop-IL}). In $\ILb$, any formula $A$ is equivalent to $\bang A$, i.e. we have $\bang A \lequiv A$. In $\ILL$, however, we do not have this equivalence in general, but we do have $\bang \bang A \lequiv \, \bang A$. 



\subsection{Girard translations}\label{gt}

\newcommand{\Girard}[1]{\textup{Gf}_{#1}}

So, an immediate approach towards a translation 
of $\ILb$ into $\ILL$ is to insert $\bang$ in all subformulas of a formula, including the formula itself. As with the Kolmogorov translation, we can do this in two equivalent ways:

\begin{defn}[Girard full translation] \label{def-Girard-full} In this case we are working with $T(A) =\;\bang A$. Again, there are two ways to present this inductively, which we will call $\Girard{o}$ and $\Girard{i}$. In the first case we make use of outer transforms $\Transform_Q^{\Girard{o}}(A) =\;\bang Q A$, for $Q \in \{ \forall x , \exists x , \bang \}$, and $\Transform_\Box^{\Girard{o}}(A, B) =\;\bang (A \squareOp B)$, for $\Box \in \{ \cwedge, \awedge, \oplus, \lto \}$, and define the modular outer $\bang$-translation $A^{\Girard{o}}$ inductively as
	\[
	\begin{array}{rclrcl}
		\Trans{(A \cwedge B)}{\Girard{o}} 
		& \pdefin & \bang (\Trans{A}{\Girard{o}} \cwedge \Trans{B}{\Girard{o}}) &
		\Trans{P}{\Girard{o}} 
		& \pdefin & \bang P, 
		\textup{~for $P$ atomic} \\[1mm]
		\Trans{(A \awedge B)}{\Girard{o}} 
		& \pdefin & \bang (\Trans{A}{\Girard{o}} \awedge \Trans{B}{\Girard{o}}) & 
		\Trans{(\forall x A)}{\Girard{o}}
		& \pdefin & \bang \forall x \Trans{A}{\Girard{o}} \\[1mm]
		\Trans{(A \oplus B)}{\Girard{o}}
		& \pdefin & \bang (\Trans{A}{\Girard{o}} \oplus \Trans{B}{\Girard{o}}) &
		\Trans{(\exists x A)}{\Girard{o}}
		& \pdefin & \bang \exists x \Trans{A}{\Girard{o}} \\[1mm]
		\Trans{(A \lto B)}{\Girard{o}} 
		& \pdefin & \bang (\Trans{A}{\Girard{o}} \lto \Trans{B}{\Girard{o}}) \quad \quad &
		\Trans{(\bang A)}{\Girard{o}} 
		& \pdefin & \bang \bang \Trans{A}{\Girard{o}}.
	\end{array}
	\]
	Alternatively, we can make use of inner transforms $\Transform_Q^{\Girard{i}}(A) = Q \bang A$, for $Q \in \{ \forall x , \exists x , \bang \}$, and $\Transform_\Box^{\Girard{i}}(A, B) =\;\bang A \squareOp \bang B$, for $\Box \in \{ \cwedge, \awedge, \oplus, \lto \}$, so that we obtain the modular inner $\bang$-translation
	\eqleft{\Trans{A}{\Girard{i}} \pdefin\;\bang \Tr{A}{\Girard{i}}}
	where $\Tr{A}{\Girard{i}}$ is defined inductively as:
	\[
	\begin{array}{rclrcl}
		\Tr{(A \cwedge B)}{\Girard{i}} 
		& \pdefin & \bang \Tr{A}{\Girard{i}} \cwedge \; \bang \Tr{B}{\Girard{i}} &
		\Tr{P}{\Girard{i}} 
		& \pdefin & P, 
		\textup{~for $P$ atomic} \\[1mm]
		\Tr{(A \awedge B)}{\Girard{i}} 
		& \pdefin & \bang \Tr{A}{\Girard{i}} \awedge \bang \Tr{B}{\Girard{i}} & 
		\Tr{(\forall x A)}{\Girard{i}}
		& \pdefin & \forall x \bang \Tr{A}{\Girard{i}} \\[1mm]
		\Tr{(A \oplus B)}{\Girard{i}}
		& \pdefin & \bang \Tr{A}{\Girard{i}} \oplus \; \bang \Tr{B}{\Girard{i}} &
		\Tr{(\exists x A)}{\Girard{i}}
		& \pdefin & \exists x \bang \Tr{A}{\Girard{i}} \\[1mm]
		\Tr{(A \lto B)}{\Girard{i}} 
		& \pdefin & \bang \Tr{A}{\Girard{i}} \lto \; \bang \Tr{B}{\Girard{i}} \quad \quad &
		\Tr{(\bang A)}{\Girard{i}} 
		& \pdefin & \bang \bang \Tr{A}{\Girard{i}}.
	\end{array}
	\]
\end{defn}




%





These two translations are easily seen to be equivalent since $\Trans{A}{\Girard{o}}$ is actually syntactically equal to $\Trans{A}{\Girard{i}}$. Not surprisingly, these are `unpolished' translations. Again it is possible to simplify them, by removing some \emph{bangs}, obtaining two other simpler but equivalent translations.

\newcommand{\GvTrans}[1]{\Trans{#1}{\circ}}
\newcommand{\GvTr}[1]{\Tr{#1}{\circ}}
\newcommand{\GnTrans}[1]{\Trans{#1}{*}}
\newcommand{\GnTr}[1]{\Tr{#1}{*}}

\begin{defn}[The call-by-value translation $\GnTrans{(\cdot)}$, \cite{GirTCS}] \label{star} Girard's call-by-value translation $\GnTrans{(\cdot)}$ is a modular outer $\bang$-translation defined by $\bang$-based outer transforms
	\[
	\begin{array}{rclrcl}
		\IT{*}{\cwedge}(A, B) 
		& \pdefin & A \cwedge B &  &  & \\[1mm]
		\IT{*}{\awedge}(A, B) 
		& \pdefin & \bang (A \awedge B) & 
		\IT{*}{\forall x} (A)
		& \pdefin & \bang \forall x A \\[1mm]
		\IT{*}{\oplus}(A, B)
		& \pdefin & A \oplus B \quad &
		\IT{*}{\exists x} (A)
		& \pdefin & \exists x A \\[1mm]
		\IT{*}{\lto}(A, B) 
		& \pdefin & \bang (A \lto B) \quad &
		\IT{*}{!}(A) 
		& \pdefin & \bang A.
	\end{array}
	\]
	Hence, $\GnTrans{A}$ is defined inductively as:
	\[
	\begin{array}{rclrcl}
		\GnTrans{(A \cwedge B)} & \pdefin & \GnTrans{A} \cwedge \GnTrans{B} & 
		\GnTrans{P} & \pdefin & \bang P, \textup{~for $P$ atomic} \\[1mm]
		\GnTrans{(A \awedge B)} & \pdefin & \bang (\GnTrans{A} \awedge \GnTrans{B})
		&
		\GnTrans{(\forall x A)} & \pdefin & \bang  \forall x \GnTrans{A} \\[1mm]
		\GnTrans{(A \oplus B)} & \pdefin & \GnTrans{A} \oplus \GnTrans{B}
		&
		\GnTrans{(\exists x A)} & \pdefin & \exists x \GnTrans{A} \\[1mm]
		\GnTrans{(A \lto B)} & \pdefin & \bang (\GnTrans{A} \lto \GnTrans{B}) \quad
		&
		\GnTrans{(\bang A)} & \pdefin & \bang \GnTrans{A}.
	\end{array}
	\]
\end{defn}

%
%
%
%
%


\begin{defn}[The call-by-name translation $\GvTrans{(\cdot)}$, \cite{GirTCS}] \label{circ} Girard's call-by-name translation is a modular inner $\bang$-translation defined by $\bang$-based inner transforms
	\[
	\begin{array}{rclrcl}
		\IT{\circ}{\cwedge}(A, B) 
		& \pdefin & \bang A \, \cwedge \, \bang B &  &  & \\[1mm]
		\IT{\circ}{\awedge}(A, B) 
		& \pdefin & A \awedge B & 
		\IT{\circ}{\forall x} (A)
		& \pdefin & \forall x A \\[1mm]
		\IT{\circ}{\oplus}(A, B)
		& \pdefin & \bang A \, \oplus \, \bang B \quad &
		\IT{\circ}{\exists x} (A)
		& \pdefin & \exists x \bang A \\[1mm]
		\IT{\circ}{\lto}(A, B) 
		& \pdefin & \bang A \lto B \quad \quad &
		\IT{\circ}{!}(A) 
		& \pdefin & \bang A.
	\end{array}
	\]
	Hence, $\GvTrans{A} =~\bang \GvTr{A}$, where $\GvTr{A}$ is defined inductively as:
	\[
	\begin{array}{rclrcl}
		\GvTr{(A \cwedge B)} & \pdefin & \bang \GvTr{A} \cwedge \, \bang \GvTr{B} &
		\GvTr{P} & \pdefin & P, \textup{~for $P$ atomic} \\[1mm]
		\GvTr{(A \awedge B)} & \pdefin & \GvTr{A} \awedge \GvTr{B}
		&
		\GvTr{(\forall x A)} & \pdefin & \forall x \GvTr{A} \\[1mm]
		\GvTr{(A \oplus B)} & \pdefin & \bang \GvTr{A} \oplus \, \bang \GvTr{B}
		&
		\GvTr{(\exists x A)} & \pdefin & \exists x \bang \GvTr{A} \\[1mm]
		\GvTr{(A \lto B)} & \pdefin & \bang \GvTr{A} \lto \GvTr{B}
		\quad \quad &
		\GvTr{(\bang A)} & \pdefin & \bang \GvTr{A}.
	\end{array}
	\]
\end{defn}

As in Proposition \ref{prop-sound-double-negation-trans}, we can also see that the above translations are all sound translations of $\ILb$ to $\ILL$. For instance, in the case of Girard's full translation, given a multiset of formulas $\Gamma$ let us denote by $\Trans{\Gamma}{\Girard{o}}$ the result of applying the transformation in question to each formula in $\Gamma$. In order to show that the translation is sound we show by induction on the structure of the proof that if $\Gamma \vdash_{\ILb} A$ then $\Trans{\Gamma}{\Girard{o}} \vdash_{\ILL} \Trans{A}{\Girard{o}}$. For instance, the case of the axiom $\Gamma,0 \vdash_{\ILb} A$, we get $\Trans{\Gamma}{\Girard{o}}, \Trans{0}{\Girard{o}} \vdash_{\ILL} \Trans{A}{\Girard{o}}$ due to the $\ILL$-equivalence $0 \lequiv \; !0$. For the case of ($\otimes$R) we have by the induction hypothesis that $\Trans{\Gamma}{\Girard{o}} \vdash_{\ILL} \Trans{A}{\Girard{o}}$ and  $\Trans{\Delta}{\Girard{o}} \vdash_{\ILL} \Trans{B}{\Girard{o}}$ and hence by ($\otimes$R) that  $\Trans{\Gamma}{\Girard{o}}, \Trans{\Delta}{\Girard{o}} \vdash_{\ILL}\Trans{A} {\Girard{o}} \otimes\Trans{B} {\Girard{o}}$. But, since $\Trans{A} {\Girard{o}}$ and $\Trans{B} {\Girard{o}}$ begin with a $!$, we can use Proposition \ref{bang-equivalences} $(i)$ to obtain $\Trans{\Gamma}{\Girard{o}}, \Trans{\Delta}{\Girard{o}} \vdash_{\ILL} !(\Trans{A} {\Girard{o}} \otimes\Trans{B} {\Girard{o}}) = \Trans{(A \otimes B)}{\Girard{o}}$ as desired. Instances of $\PRO$ can also be seen to be derivable in $\ILL$ due to the extra bangs introduced by the $\Trans{(\cdot)}{\Girard{o}}$-translation.

\subsection{$\GnTrans{(\cdot)}$ is a simplification from outside of $\Trans{(\cdot)}{\Girard{o}}$}\label{so}

Let us first show that $\GnTrans{(\cdot)}$ is a simplification from outside of $\Trans{(\cdot)}{\Girard{o}}$, in a similar way that the G\"odel-Gentzen translation $\Trans{(\cdot)}{\Goedel}$ is a simplification from outside of Kolmogorov's $\Trans{(\cdot)}{\Kolm{o}}$ translation. 

\begin{prop} \label{outsidebang} [$\GnTrans{(\cdot)}$ is a simplification from outside of $\Trans{(\cdot)}{\Girard{o}}$] For any formula $A$ of $\IL$, one can obtain $\GnTrans{A}$ from $\Trans{A}{\Girard{o}}$ by systematically applying the following formula reductions starting from the whole formula and inductively applying these to the subformulas:
	\[
	\begin{array}{rcl}
		\underbrace{\bang (\bang A \cwedge \bang B)}_{ \IT{\Girard{o}}{\cwedge}(\bang A, \bang B)} 
		& \quad \mapsto \quad & 
		\underbrace{\bang A \cwedge \bang B}_{\IT{*}{\cwedge}(\bang A, \bang B)} \\[3mm]
		\underbrace{\bang (\bang A \oplus \bang B)}_{ \IT{\Girard{o}}{\oplus}(\bang A, \bang B)} 
		& \quad \mapsto \quad & 
		\underbrace{ \bang A \oplus \bang B}_{\IT{*}{\oplus}(\bang A, \bang B)} \\[3mm]
  \underbrace{\bang \exists x \bang A}_{\IT{\Girard{o}}{\exists x}(\bang A)}
		& \quad \mapsto \quad &
		\underbrace{ \exists x \bang A}_{ \IT{*}{\exists x}(\bang A)}\\[3mm]
		\underbrace{\bang \bang \bang A}_{ \IT{\Girard{o}}{!}(\bang A)} 
		& \quad \mapsto \quad & 
		\underbrace{\bang \bang A}_{\IT{*}{!}(\bang A)}		
	\end{array}
	\]
	and, moreover, by Proposition \ref{bang-equivalences}, the reductions above are reversible in $\ILL$. 
\end{prop}

\begin{proof} Let us write $A \mapsto^o B$ when $B$ is obtained from $A$ by systematically applying the reductions $\mapsto$ listed above ``from the outside''. Similar to the proof of Proposition \ref{thm-goedel-simplifies-kol}, we can show by induction on $A$ that $A^{\Girard{o}} \mapsto^o \GnTrans{A}$. 
\end{proof}

\subsection{$\GvTrans{(\cdot)}$ is a simplification from inside of $\Trans{(\cdot)}{\Girard{i}}$} \label{si}

If, on the other hand, we start with the inner presentation of the Girard translation $\Trans{(\cdot)}{\Girard{i}}$, and systematically simplify it from 'inside', we obtain the call-by-name translation $\GvTrans{(\cdot)}$. 

\begin{prop}\label{simpinside}[$\GvTrans{(\cdot)}$ is a simplification from inside of $\Trans{(\cdot)}{\Girard{i}}$] For any formula $A$, one can obtain $\GvTrans{A}$ from $\Trans{A}{\Girard{i}}$ by systematically applying the following formula reductions starting from the atomic formulas and inductively applying these to composite formulas:
	\[
	\begin{array}{rcl}
		\underbrace{\bang (\bang A \awedge \bang B)}_{\bang \IT{\Girard{i}}{\awedge}(A, B)} 
		& \quad \mapsto \quad & 
		\underbrace{\bang (A \awedge B)}_{\bang \IT{\circ}{\awedge}(A, B)} \\[2mm]
		\underbrace{\bang (\bang A \lto \, \bang B)}_{\bang \IT{\Girard{i}}{\lto}(A, B)} 
		& \quad \mapsto \quad & 
		\underbrace{\bang (\bang A \lto B)}_{\bang \IT{\circ}{\lto}(A, B)} \\[2mm]
  \underbrace{\bang \forall x \bang A}_{\bang \IT{\Girard{i}}{\forall x}(A)}
		& \quad \mapsto \quad &
		\underbrace{\bang \forall x A}_{\bang \IT{\circ}{\forall x}(A)} \\[2mm]
		\underbrace{\bang \bang \bang A}_{\bang \IT{\Girard{i}}{! }(A)} 
		& \quad \mapsto \quad & 
		\underbrace{\bang \bang A}_{\bang \IT{\circ}{! }(A)} 
	\end{array}
	\]
	and, moreover, by Proposition \ref{bang-equivalences}, the reductions above are reversible in $\ILL$. 
\end{prop}

\begin{proof} Let us write $A \mapsto^i B$ when $B$ is obtained from $A$ by systematically applying the reductions $\mapsto$ listed above `from inside'. By induction on $A$ we can show that $A^{\Girard{i}} \mapsto^i \; \bang \GvTr{A}$. 
\end{proof}

\begin{obs} From the above it follows that $\bang \GvTr{A}$ is $\ILL$-equivalent to $\GnTrans{A}$. It is interesting to compare this situation to that of the translations $(\cdot)^\circ$ and $(\cdot)^\square$ of $\IL$ into S4 which are the analogues of the translations $\GvTrans{(\cdot)}$ and $\GnTrans{(\cdot)}$. Indeed, we have \cite[p.288, Prop. 9.2.2]{basic} that $S4 \vdash \square P^{\circ} \leftrightarrow P^{\square}$.
\end{obs}

\begin{obs} In the previous study of the Girard's translations (Subsections \ref{gt}, \ref{so}, and \ref{si} above), we considered translations from $\ILb$ to $\ILL$. Note that the whole strategy works equally well if we instead consider translations from $\IL_\bot$ to $\ILZ$. This observation will be used later in Section \ref{sec-cl-to-cll}. 
\end{obs}

\section{Translations from $\CLLb$ to $\ILZ$}
\label{sec-cll-to-ill}



In Section \ref{sec-cl-to-il} we studied different negative translations of $\CLb = \ILL_\bot + \PRO + \DNE$ into $\IL_\bot = \ILL_\bot + \PRO$. In this section we will consider ``linear'' variants of these translations, i.e. translations\footnote{See \cite{Lam95} for an interesting use of a negative translation of classical into intuitionistic linear logic in the definition of a game semantics for $\CLL$. For a comprehensive study of the various negative translations of $\CLL$ into $\ILL$ see \cite{Lau18}, where, in particular, these translations are used to derive conservativity results.} of $\CLLb = \ILL_\bot + \DNE$ into $\ILL_\bot$. In the linear setting, i.e. in the absence of $\PRO$, and hence the equivalence $\bang A \lequiv A$, fewer simplifications are available, so as expected the linear variants of the G\"odel-Gentzen and the Kuroda translations will have more double negations than their standard counterparts. 

\subsection{Linear negative translations}

Let us start by observing that the (modular) Kolmogorov translations of Definition \ref{def:Kolmogorov-translations}, of $\CLb$ into $\IL_\bot$, are also translations of $\CLLb$ into $\ILL_\bot$, the only difference being that we no longer need to validate the promotion axiom schema $\PRO$.

The same cannot be said about the G\"odel-Gentzen and the Kuroda translations of Definitions \ref{def-god} and \ref{def-kur}, as these fail in the linear logic setting. An algebraic analysis of the G\"odel-Gentzen and the Glivenko translations has been given \cite{ao}, and we can obtain some useful counter-models from there. For instance, even though $\neg \neg (P \otimes Q) \lto P \otimes Q$ is provable in $\CLLb$, it is not the case that its G\"odel-Gentzen translation $\neg \neg (\neg \neg P \otimes \neg \neg Q) \lto \neg \neg P \otimes \neg \neg Q$ is provable in $\ILL_\bot$. Similarly, $\neg \neg P \lto P$ is provable in $\CLLb$, but its Kuroda translation $\neg \neg (\neg \neg P \lto P)$ is not provable in $\ILL_\bot$.

In order for these translations to work in the absence of $\PRO$, we need to add extra double negations, leading us to the following ``linear'' variants:


\begin{defn}[Linear Gödel-Gentzen translation] The linear G\"odel-Gentzen negative translation is a modular outer $\neg \neg$-translation defined by $\neg \neg$-based outer transforms:
	\[
	\begin{array}{cclccl}
		\IT{\lGoedel}{\cwedge}(A, B) 
		& \pdefin & \neg \neg (A \cwedge B) &  &  & \\[1mm]
		\IT{\lGoedel}{\awedge}(A, B) 
		& \pdefin & A \awedge B & 
		\IT{\lGoedel}{\forall x} (A)
		& \pdefin & \forall x A \\[1mm]
		\IT{\lGoedel}{\oplus}(A, B)
		& \pdefin & \neg \neg (A \oplus B) \quad &
		\IT{\lGoedel}{\exists x} (A)
		& \pdefin & \neg \neg \exists x A \\[1mm]
		\IT{\lGoedel}{\lto}(A, B) 
		& \pdefin & A \lto B \quad \quad &
		\IT{\lGoedel}{! }(A) 
		& \pdefin & \neg \neg \bang A.
	\end{array}
	\]
	Hence, $\Trans{A}{\lGoedel}$ is defined inductively as:
	\[
	\begin{array}{rclrcl}
		\Trans{(A \cwedge B)}{\lGoedel} 
		& \pdefin & \neg \neg (\Trans{A}{\lGoedel} \cwedge \Trans{B}{\lGoedel}) &
		\Trans{P}{\lGoedel} 
		& \pdefin & \neg \neg P, 
		\textup{~for $P$ atomic} \\[1mm]
		\Trans{(A \awedge B)}{\lGoedel} 
		& \pdefin & \Trans{A}{\lGoedel} \awedge \Trans{B}{\lGoedel} & 
		\Trans{(\forall x A)}{\lGoedel}
		& \pdefin & \forall x \Trans{A}{\lGoedel} \\[1mm]
		\Trans{(A \oplus B)}{\lGoedel}
		& \pdefin & \neg \neg (\Trans{A}{\lGoedel} \oplus \Trans{B}{\lGoedel})  \quad &
		\Trans{(\exists x A)}{\lGoedel}
		& \pdefin & \neg \neg \exists x \Trans{A}{\lGoedel} \\[1mm]
		\Trans{(A \lto B)}{\lGoedel} 
		& \pdefin & \Trans{A}{\lGoedel} \lto \Trans{B}{\lGoedel} &
		\Trans{(\bang A)}{\lGoedel} 
		& \pdefin & \neg \neg \bang \Trans{A}{\lGoedel}.
	\end{array}
	\]
\end{defn}

The Kuroda negative translation also has a linear variant, which is similarly ``less optimal'' (extra double negation in the clauses for additive conjunction, implication and modality) than its standard counterpart of Definition \ref{def-kur}:

\begin{defn}[Linear Kuroda translation] The linear Kuroda negation translation is a modular inner $\neg \neg$-translation defined by $\neg \neg$-based inner transforms:
	\[
	\begin{array}{rclrcl}
		\IT{\lKuroda}{\cwedge}(A, B) 
		& \pdefin & A \cwedge B &  &  & \\[1mm]
		\IT{\lKuroda}{\awedge}(A, B) 
		& \pdefin & \neg \neg A \awedge \neg \neg B \quad & 
		\IT{\lKuroda}{\forall x} (A)
		& \pdefin & \forall x \neg \neg A \\[1mm]
		\IT{\lKuroda}{\oplus}(A, B)
		& \pdefin & A \oplus B &
		\IT{\lKuroda}{\exists x} (A)
		& \pdefin & \exists x A \\[1mm]
		\IT{\lKuroda}{\lto}(A, B) 
		& \pdefin & A \lto \neg \neg B &
		\IT{\lKuroda}{! }(A) 
		& \pdefin & \bang \neg \neg A.
	\end{array}
	\]
	Hence, $\Trans{A}{\lKuroda} \pdefin \neg \neg \Tr{A}{\lKuroda}$, where $\Tr{A}{\lKuroda}$ is defined inductively as:
	\[
	\begin{array}{rclrcl}
		\Tr{(A \cwedge B)}{\lKuroda} 
		& \pdefin & \Tr{A}{\lKuroda} \cwedge \Tr{B}{\lKuroda} &
		\Tr{(P)}{\lKuroda} 
		& \pdefin & P, 
		\textup{~for $P$ atomic} \\[1mm]
		\Tr{(A \awedge B)}{\lKuroda} 
		& \pdefin & \neg \neg \Tr{A}{\lKuroda} \awedge \neg \neg \Tr{B}{\lKuroda} \quad & 
		\Tr{(\forall x A)}{\lKuroda}
		& \pdefin & \forall x \neg \neg \Tr{A}{\lKuroda} \\[1mm]
		\Tr{(A \oplus B)}{\lKuroda}
		& \pdefin & \Tr{A}{\lKuroda} \oplus \Tr{B}{\lKuroda} &
		\Tr{(\exists x A)}{\lKuroda}
		& \pdefin & \exists x \Tr{A}{\lKuroda} \\[1mm]
		\Tr{(A \lto B)}{\lKuroda} 
		& \pdefin & \Tr{A}{\lKuroda} \lto \neg \neg \Tr{B}{\lKuroda} &
		\Tr{(\bang A)}{\lKuroda} 
		& \pdefin & \bang \neg \neg \Tr{A}{\lKuroda}.
	\end{array}
	\]
\end{defn}

\begin{prop}[Linear negative translations] The two Kolmogorov translations (Definition \ref{def:Kolmogorov-translations}), the linear G\"odel-Gentzen, and the linear Kuroda translations are all sound translations from $\CLLb$ to $\ILL_\bot$.
\end{prop}

\begin{proof} Again, as in Proposition \ref{prop-sound-double-negation-trans}, these translations are easily shown to be sound by induction on the structure of proofs. For instance, for the linear G\"odel-Gentzen translation, suppose we have the case of the axiom $\Gamma, 0 \vdash_{\CLLb} A$.  We need to show that $\Trans{\Gamma}{\lGoedel}, \neg\neg 0 \vdash_{\ILL_\bot} \Trans{A}{\lGoedel}  $.  Note first that in $\ILL_\bot$ if we have $\Gamma,A\vdash_{\ILL_\bot} B$ then $\Gamma, \neg A\vdash_{\ILL_\bot} \neg A$ and hence $\Gamma,\neg\neg A\vdash_{\ILL_\bot} \neg\neg B$. Since $\neg\neg 0 = (0\lto \bot)\lto \bot $ this follows from applying the above fact to  $\Trans{\Gamma}{\lGoedel},  0 \vdash_{\ILL_\bot} \Trans{A}{\lGoedel}$ and noting that by Proposition 4 we can show that $\Trans{A}{\lGoedel} \lequiv \neg\neg A'$ for some formula $A'$ and hence that $\neg\neg \Trans{A}{\lGoedel} \lequiv \Trans{A}{\lGoedel}$.  The case of the rule $\otimes$R for instance is shown applying Proposition \ref{not-not-equivalences} $(i)$ and using again the fact that $\Trans{A}{\lGoedel}$ is equivalent to an expression of the form  $\neg\neg A'$, for some formula $A'$.  
\end{proof}

One might ask, how does one know which connectives require these extra double negations? We believe the answer lies in the following results.

\subsection{$\Trans{(\cdot)}{\lGoedel}$ is a simplification from outside of $\Trans{(\cdot)}{\Kolm{o}}$}


\begin{prop}[$\Trans{(\cdot)}{\lGoedel}$ is a simplification from outside of $\Trans{(\cdot)}{\Kolm{o}}$] For any formula $A$, one can obtain $\Trans{A}{\lGoedel}$ from $\Trans{A}{\Kolm{o}}$ by systematically applying the following formula reductions starting from the whole formula and inductively applying these to the subformulas:
	\[
	\begin{array}{rcl}
		\underbrace{\neg \neg (\neg \neg A \awedge \neg \neg B)}_{\IT{\Kolm{o}}{\awedge}(\neg \neg A, \neg \neg B)} 
		& \quad \mapsto \quad & 
		\underbrace{\neg \neg A \awedge \neg \neg B}_{\IT{\lGoedel}{\awedge}(\neg \neg A, \neg \neg B)} \\[3mm]
		 \underbrace{\neg \neg (\neg \neg A \lto \neg \neg B)}_{\IT{\Kolm{o}}{\lto}(\neg \neg A, \neg \neg B)} 
		     & \quad \mapsto \quad & 
		     \underbrace{\neg \neg A \lto \neg \neg B}_{\IT{\lGoedel}{\lto}(\neg \neg A, \neg \neg B)} \\[3mm]
		\underbrace{\neg \neg \forall x \neg \neg A}_{\IT{\Kolm{o}}{\forall x}(\neg \neg A)}
		& \quad \mapsto \quad &
		\underbrace{\forall x \neg \neg A}_{\IT{\lGoedel}{\forall x}(\neg \neg A)}.
	\end{array}
	\]
	%
    Hence, $\Trans{(\cdot)}{\lGoedel}$ is a simplification from outside of $\Trans{(\cdot)}{\Kolm{o}}$, in the sense of Definition \ref{def-simp-from-outside}. Moreover, by Proposition \ref{not-not-equivalences}, the reductions above are reversible in $\ILZ$. 
\end{prop}

The proof of the above proposition would be similar to that of Proposition \ref{thm-goedel-simplifies-kol}, except that in here we are working in $\ILZ$, instead of $\IL_\bot$, so fewer equivalences (or simplifications) would be available.


\subsection{$\Trans{(\cdot)}{\lKuroda}$ is a simplification from inside of $\Trans{(\cdot)}{\Kolm{i}}$}

A similar phenomenon holds for the inner presentation of the Kolmogorov translation $\Trans{(\cdot)}{\Kolm{i}}$ and Kuroda's linear translation $\Trans{(\cdot)}{\Kolm{i}}$, where we get a linear version of Proposition \ref{prop-kur-simplifies-kolm}:

\begin{prop}[$\Trans{(\cdot)}{\lKuroda}$ is a simplification from inside of $\Trans{(\cdot)}{\Kolm{i}}$] For any formula $A$, one can obtain $\Trans{A}{\lKuroda}$ from $\Trans{A}{\Kolm{i}}$ by systematically applying the following formula reductions starting from the atomic formulas and inductively applying these to composite formulas:
	\[
	\begin{array}{rcl}
		\underbrace{\neg \neg (\neg \neg A \cwedge \neg \neg B)}_{\neg \neg \IT{\Kolm{i}}{\cwedge}(A, B)} 
		& \quad \mapsto \quad & 
		\underbrace{\neg \neg (A \cwedge B)}_{\neg \neg \IT{\lKuroda}{\cwedge}(A, B)} \\[3mm]
		\underbrace{\neg \neg (\neg \neg A \oplus \neg \neg B)}_{\neg \neg \IT{\Kolm{i}}{\oplus}(A, B)} 
		& \quad \mapsto \quad & 
		\underbrace{\neg \neg (A \oplus B)}_{\neg \neg \IT{\lKuroda}{\oplus}(A, B)} \\[3mm]
		\underbrace{\neg \neg (\neg \neg A \lto \neg \neg B)}_{\neg \neg \IT{\Kolm{i}}{\lto}(A, B)} 
		& \quad \mapsto \quad & 
		\underbrace{\neg \neg (A \lto \neg \neg B)}_{\neg \neg \IT{\lKuroda}{\lto}(A, B)} \\[3mm]
		\underbrace{\neg \neg \exists x \neg \neg A}_{\neg \neg \IT{\Kolm{i}}{\exists x}(A)}
		& \quad \mapsto \quad &
		\underbrace{\neg \neg \exists x A}_{\neg \neg \IT{\lKuroda}{\exists x}(A)}.
	\end{array}
	\]
	%
    Hence, $\Trans{(\cdot)}{\lKuroda}$ is a simplification from inside of $\Trans{(\cdot)}{\Kolm{i}}$, in the sense of Definition \ref{def-simp-from-inside}. Moreover, by Proposition \ref{not-not-equivalences}, the reductions above are reversible in $\ILZ$. 
\end{prop}

\section{Translations from $\CLb$ to $\CLLb$}
\label{sec-cl-to-cll}

We can now derive translations from $\CLb$ to $\CLLb$ as compositions of the translations from $\CLb$ to $\IL_\bot$ with those from $\IL_\bot$ to $\ILL_{\bot}$ (via the inclusion of $\ILL_{\bot}$ into $\CLLb$):
\[ 
    \CLb \quad \stackrel{\scriptsize{\mbox{Section \ref{sec-cl-to-il}}}}{\longmapsto} \quad 
    \IL_\bot \quad \stackrel{\scriptsize{\mbox{Section \ref{sec-il-to-ill}}}}{\longmapsto} \quad 
    \ILL_\bot \quad \stackrel{\subset} {\longmapsto} \quad \CLLb.
\]    
%
We are going to consider the compositions of the Kuroda $\Trans{(\cdot)}{\Kuroda}$ and Gödel $\Trans{(\cdot)}{\Goedel}$ translations with the Girard $(\cdot)^*$ and $(\cdot)^\circ$ translations, obtaining four translations of $\CLb$ to $\CLLb$. 




Let $\whynot A$ be an abbreviation for $\neg \bang \neg A$. We will make use of the following result about the Girard translations $(\cdot)^*$ and $(\cdot)^\circ$ :

\begin{lem} Let $(\cdot)^{*}$ and $(\cdot)^{\circ}$ be the Girard translations presented in Definitions \ref{star} and \ref{circ}. For any formula $A$ the following are provable in $\CLLb$
	\begin{itemize}
		\item[$(i)$] $(\neg \neg A)_\circ \lequiv \;  \whynot \bang A_\circ$, and hence $(\neg \neg A)^\circ \lequiv \; \bang  \whynot \bang A_\circ$, and
		\item[$(ii)$] $(\neg \neg A)^* \lequiv \; \bang  \whynot  A^*$.
	\end{itemize}
\end{lem}

\begin{proof} $(i)$ A direct computation yields: 
\[ 
    (\neg\neg A)_\circ 
    = \; \bang ( \bang A_\circ \lto \bot ) \lto \bot 
    = \neg \bang  \neg  \bang   A_\circ  
    = \; \whynot \bang  A_\circ
\] 
which implies $(\neg \neg A)^\circ \lequiv \; \bang  \whynot \bang A_\circ$. $(ii)$ We have: 
\[ 
    (\neg\neg A)^* 
    = \, \bang  ( \bang  (A^* \lto \bot^* ) \lto \bot^*) 
    = \, \bang  ( \bang  (A^* \lto ~ \bang \bot ) \lto ~ \bang \bot).
\]
Using $(x)$ of Proposition 5  we get that 
%
\[ 
    \bang  ( \bang  (A^* \lto ~ \bang \bot ) \lto ~ \bang \bot) 
    \lequiv \, \bang  ( \bang  (A^* \lto  \bot ) \lto  \bot) 
    = \, \bang  \neg \bang  \neg A^* 
    = \, \bang   \whynot   A^* 
\]
which concludes the proof. 
\end{proof}

Before proceeding we state two more lemmas about $\CLL_b  = \ILL_\bot + \DNE$ which will be need later.

\begin{lem}\label{admissible} The following are derivable rules in $\CLLb$:
\[
\begin{array}{ccc}
    \begin{prooftree}
        \bang \Gamma, A \vdash \whynot B
        \Justifies
        \bang \Gamma, \whynot A \vdash \whynot B
    \end{prooftree}
    & \hspace{2cm} &	\begin{prooftree}
        \Gamma \vdash B
        \Justifies
        \Gamma \vdash \whynot B 
    \end{prooftree}
\end{array}
\]
We will use ``double lines'' for one or more steps in a proof.
\end{lem}
\begin{proof} These are straightforward to derive, for instance, the first rule can be derived as:
\[\begin{prooftree}
	\[
        \[
            \[
                \[
                    \[ 
                        \bang \Gamma, A \vdash \whynot B
                        \justifies
                        \bang \Gamma, A \vdash \neg \bang \neg B
                        \using{\textup{(def)}}
                    \]
                    \quad 
                    \[
                        \bang \neg B \vdash \bang \neg B \quad \bot \vdash \bot
                        \justifies
                        \bang \neg B, \neg \bang \neg B \vdash \bot
                        \using{(\lto \mbox{L})}
                    \]
                    \justifies
                    \bang \Gamma, \bang \neg B, A \vdash \bot
                    \using{(\textup{cut})}
                \]
                \justifies
	          \bang \Gamma,	\bang \neg B \vdash \neg A
                \using{(\lto\textup{R})}
           \]
	   \justifies
	   \bang \Gamma, \bang \neg B \vdash \bang \neg A
            \using{(\bang\textup{R})}
        \]
	  \Justifies
	  \bang \Gamma,	\neg \bang \neg A \vdash \neg 	\bang \neg B 
    \]
	\justifies
	\bang \Gamma, \whynot A \vdash \whynot B
    \using{\textup{(def)}}
\end{prooftree}
\]
where on the double-line step (derivable rule) we are again using (cut) and the axiom $\bot \vdash \bot$ to go from $\bang \Gamma, \bang \neg B \vdash \bang \neg A$ to $\bang \Gamma,	\neg \bang \neg A \vdash \neg 	\bang \neg B$.
The second rule can also be easily derived in $\ILL + \DNE$.
\end{proof}

\begin{lem} \label{bang \whynot } $\CLLb$ proves the following equivalences:
	\begin{enumerate}
		\item[(i)] $\bang \bang A \lequiv \, \bang A$
		\item[(ii)] $\bang \whynot \bang \whynot A \lequiv \, \bang \whynot A$
		\item[(iii)]  $\whynot \bang \whynot \bang A \lequiv \, \whynot \bang A$
	\end{enumerate}
\end{lem}

\begin{proof} $(i)$ is immediate. The two implications in $(ii)$ can be shown as
\[
\begin{prooftree}
	\[
        \[
            \[ 
                \whynot A \vdash \whynot A
            	\justifies 
            	\bang \whynot A \vdash \whynot A
                \using{(\bang\textup{L})}
            \]
        	\Justifies
        	\whynot	\bang \whynot A \vdash \whynot A
            \using {\mbox{(Lemma \ref{admissible})}}
        \]
        \justifies
        \bang \whynot	\bang \whynot A \vdash \whynot A
        \using{(\bang\textup{L})}
    \]
    \justifies
    \bang \whynot \bang \whynot A \vdash \bang \whynot A
    \using{(\bang\textup{R})}
\end{prooftree}
\quad \quad \quad
\begin{prooftree}
	\[ \bang \whynot A \vdash \bang \whynot A
	\Justifies 
	\bang \whynot A \vdash \whynot\bang \whynot A
    \using
    \using {\mbox{(Lemma \ref{admissible})}}
    \]
	\justifies
	\bang \whynot A \vdash \bang \whynot\bang \whynot A
    \using{(\bang\textup{R})} 
\end{prooftree} \]
$(iii)$ can be shown in a similar way.
\end{proof}

Let us consider first the combination of G\"odel's negative translation with the Girard $(\cdot)^\circ$ translation:

\begin{thm}[$(\cdot)^{\Goedel\circ}$ translation] \label{g-circ-composition} Consider the following translation of $\CLb$ to $\CLLb$ obtained by composing $(\cdot)^{\Goedel}$ and $(\cdot)^\circ$:
	\eqleft{A^{\Goedel\circ} \pdefin (A^{\Goedel})^\circ.} 
	This composition can be presented in a modular (and simpler) way as follows:
	\eqleft{A^{\Goedel\circ}  \pdefin ~\bang  A_{\Goedel\circ}}
	where $A_{\Goedel\circ}$ is defined inductively as:
	\[
	\begin{array}{rclrcl}
		(A\cwedge B)_{\Goedel\circ} 
		& \pdefin &  \bang  A_{\Goedel\circ} \cwedge \bang  B_{\Goedel\circ}&
		P_{\Goedel\circ} 
		& \pdefin &   \whynot  \bang P, 
		\textup{~for $P$ atomic} \\[1mm]
		(A\awedge B)_{\Goedel\circ} 
		& \pdefin & A_{\Goedel\circ} \awedge B_{\Goedel\circ} \quad & 
		(\forall x A)_{\Goedel\circ}
		& \pdefin & \forall x A_{\Goedel\circ}\\[1mm]
		(A\oplus B)_{\Goedel\circ} 
		& \pdefin &   \whynot ( \bang  A_{\Goedel\circ} \oplus \bang  B_{\Goedel\circ})  &
		(\exists x A)_{\Goedel\circ}
		& \pdefin &  \whynot  \exists x \bang A_{\Goedel\circ}  \\[1mm]
		(A\lto B)_{\Goedel\circ}
		& \pdefin &\bang  A_{\Goedel\circ} \lto B_{\Goedel\circ} &
		\Tr{(\bang A)}{\Goedel\circ} 
		& \pdefin &  \bang  \Tr{A}{\Goedel\circ}.
	\end{array}
	\]
\end{thm}

\begin{proof} We can show by induction on $A$ that $\bang (A^{\Goedel})_\circ$ is equivalent to the simpler formula $\bang  A_{\Goedel\circ}$, proving this way that the composed translation $(A^{\Goedel})^\circ$ is equivalent to $\bang  A_{\Goedel\circ}$. For $P$ atomic we have:
\[ 
    \bang (P^{\Goedel})_\circ 
    \equiv \; \bang (\neg \neg P)_\circ  
    \equiv \; \bang  \whynot \bang P 
    \equiv \; \bang P_{\Goedel\circ}.
\]
In the case of implications $A \lto B$ we have:
\[
    \bang ((A\lto B)^\Goedel)_\circ 
    \equiv \; \bang (A^{\Goedel} \lto B^{\Goedel})_\circ 
    \equiv \; \bang ( \bang  (A^{\Goedel})_\circ \lto (B^{\Goedel})_\circ)
    \stackrel{\textup{(IH)}}{\lequiv} \; \bang (\bang  A_{\Goedel\circ} \lto B_{\Goedel\circ})
\]
using $\bang (\bang A \lto B) \lequiv \; \bang (\bang A \lto \, \bang B)$, cf. Prop. \ref{bang-equivalences} $(xi)$. \\[1mm]
Additive conjunction:
\[
    \bang ((A\awedge B)^{\Goedel})_\circ 
    \equiv \; \bang ((A^{\Goedel})_\circ \awedge (B^{\Goedel})_\circ)
    \lequiv \; \bang (\bang (A^{\Goedel})_\circ \awedge \bang (B^{\Goedel})_\circ)
    \stackrel{\textup{(IH)}}{\lequiv} \; \bang (A_{\Goedel\circ} \awedge B_{\Goedel\circ})
\]
using $\bang (\bang A \awedge \bang B) \lequiv \; \bang (A \awedge B)$, cf. Prop. \ref{bang-equivalences} $(ix)$. \\[1mm]
For the additive disjunction we have:
\[
    \bang ((A\oplus B)^{\Goedel})_{\circ} 
    \equiv \; \bang (\neg \neg(A^{\Goedel} \oplus B^{\Goedel}))_\circ 
    \equiv \; \bang  \whynot \bang (\bang  (A^{\Goedel})_\circ \oplus \bang  (B^{\Goedel})_\circ) 
    \stackrel{\textup{(IH)}}{\lequiv} \; \bang  \whynot (\bang  A_{\Goedel\circ} \oplus \bang  B_{\Goedel\circ})
\]
using $\bang (\bang A \oplus \bang B) \lequiv \; \bang A \oplus \bang B$, cf. Prop. \ref{bang-equivalences} $(iii)$. \\[1mm]
Multiplicative conjunction:
\[
    \bang ((A\cwedge B)_\Goedel)^\circ 
    \equiv \; \bang (\bang (A^{\Goedel})_\circ \cwedge \bang (B^{\Goedel})_\circ) 
    \stackrel{\textup{(IH)}}{\lequiv} \; \bang (\bang  A_{\Goedel\circ} \cwedge \bang B_{\Goedel\circ}).
\]
Universal quantifier:
\[
    \bang ((\forall x A)^\Goedel)_\circ 
    \equiv \; \bang \forall x (A^{\Goedel})_\circ
    \stackrel{\textup{(IH)}}{\lequiv} \; \bang \forall x A_{\Goedel\circ}
\]
using $\bang \forall x \bang A \lequiv \; \bang \forall x A$, cf. Prop. \ref{bang-equivalences} $(xii)$. \\[1mm]
Existential quantifier:
\[
    \bang ((\exists x A)^\Goedel)_\circ 
    \equiv \; \bang (\neg \neg \exists x A^{\Goedel})_\circ 
    \equiv \; \bang \whynot \bang \exists x \bang  (A^{\Goedel})_\circ
    \stackrel{\textup{(IH)}}{\lequiv} \; \bang \whynot \exists x \bang A_{\Goedel\circ}
\]
using $\bang \exists x \bang A \lequiv \exists x \bang A$, cf. Prop. \ref{bang-equivalences} $(vi)$. \\[1mm]
Modality $\bang$
\[
    \bang ((\bang A)^\Goedel)_\circ 
    \equiv \; \bang ( \bang A^{\Goedel})_\circ 
    \equiv \; \bang \bang (A^{\Goedel})_\circ
    \stackrel{\textup{(L.\ref{bang \whynot})}}{\lequiv} \; \bang (A^{\Goedel})_\circ \stackrel{\textup{(IH)}}{\lequiv} \;  \bang A_{\Goedel\circ}
\]
which concludes the proof.
\end{proof}


Now consider the composition of the G\"odel negative translation $(\cdot)^{\Goedel}$ with the $(\cdot)^*$ translation.

\begin{thm}[$(\cdot)^{\Goedel*}$ translation] Consider the following translation of $\CLb$ to $\CLLb$ obtained by composing $(\cdot)^{\Goedel}$ and $(\cdot)^*$:
	\eqleft{A^{\Goedel*} \pdefin (A^{\Goedel})^*.} 
	This composition can be presented directly in a modular (and simpler) way as follows:
	\[
	\begin{array}{rclrcl}
		(A\cwedge B)^{\Goedel*} 
		& \pdefin &  A^{\Goedel*} \cwedge  B^{\Goedel*}&
		P^{\Goedel*} 
		& \pdefin &  \bang   \whynot  \bang P, 
		\textup{~for $P$ atomic} \\[1mm]
		(A\awedge B)^{\Goedel*} 
		& \pdefin & \bang  (A^{\Goedel*} \awedge B^{\Goedel*}) \quad & 
		(\forall x A)^{\Goedel*}
		& \pdefin & \bang \forall x A^{\Goedel*}\\[1mm]
		(A\oplus B)^{\Goedel*} 
		& \pdefin & \bang \whynot (A^{\Goedel*} \oplus B^{\Goedel*}) &
		(\exists x A)^{\Goedel*}
		& \pdefin & \bang \whynot \exists x A^{\Goedel*}  \\[1mm]
		(A\lto B)^{\Goedel*}
		& \pdefin &\bang  (A^{\Goedel*} \lto B^{\Goedel*}) &
		(\bang A)^{\Goedel*} 
		& \pdefin &    \bang   A^{\Goedel*}.
	\end{array}
	\]
\end{thm}

\begin{proof} A simple induction on the structure of $A$ will show that the direct modular translation $(\cdot)^{\Goedel*}$ is equivalent to the composition $((\cdot)^{\Goedel})^*$.
%
For instance, we have
$(\exists x A)^{\Goedel*} \equiv (\neg \neg \exists x A^{\Goedel})^* \stackrel{\textup{(IH)}}{\lequiv} \; \bang \whynot \exists x A^{\Goedel*}$. 
\end{proof}





Similarly, we can consider the composition of the Kuroda translation $(\cdot)^{\Kuroda}$ with both $(\cdot)^\circ$ and $(\cdot)^*$.

\begin{thm}[$(\cdot)^{\Kuroda\circ}$ translation] Consider the following translation from $\CLb$ to $\CLLb$ obtained by composing $(\cdot)^{\Kuroda}$ and $(\cdot)^\circ$:
\eqleft{A^{\Kuroda\circ}  \pdefin (A^{\Kuroda})^\circ.}
This composition can be presented directly in a modular (and simpler) way as follows:
\eqleft{A^{\Kuroda\circ}  \pdefin ~\bang   \whynot   \bang  A_{\Kuroda\circ}}
where $\Tr{A}{\Kuroda\circ}$ is defined inductively as:
\[
\begin{array}{rclrcl}
    (A\cwedge B)_{\Kuroda\circ} 
    & \pdefin &  \bang  A_{\Kuroda\circ} \, \cwedge \, \bang  B_{\Kuroda\circ}&
    P_{\Kuroda\circ} 
    & \pdefin & P, 
    \textup{~for $P$ atomic} \\[1mm]
    (A \awedge B)_{\Kuroda\circ} 
    & \pdefin &     \bang  A_{\Kuroda\circ} \awedge \bang  B_{\Kuroda\circ} \quad & 
    (\forall x A)_{\Kuroda\circ}
    & \pdefin & \forall x \whynot \bang  A_{\Kuroda\circ}\\[1mm]
    (A\oplus B)_{\Kuroda\circ} 
    & \pdefin &   \bang  A_{\Kuroda\circ} \oplus \bang  B_{\Kuroda\circ}  &
    (\exists x A)_{\Kuroda\circ}
    & \pdefin & \exists x  \bang A_{\Kuroda\circ}  \\[1mm]
    (A\lto B)_{\Kuroda\circ}
    & \pdefin &\bang  A_{\Kuroda\circ} \lto B_{\Kuroda \circ} &
    \Tr{(\bang A)}{\Kuroda\circ} 
    & \pdefin &   \bang   \Tr{A}{\Kuroda\circ}.
\end{array}
\]
\end{thm}
\begin{proof} As in the proof of Theorem \ref{g-circ-composition}, we can show by induction on $A$ that $\bang \whynot \bang (A_{\Kuroda})_\circ$ (which is $(A^{\Kuroda})^\circ$) is equivalent to $\bang \whynot \bang  A_{\Kuroda\circ}$ (which is $A^{\Kuroda\circ}$). In fact, we can prove the stronger result that $\bang (A_{\Kuroda})_\circ$ is equivalent to $\bang  A_{\Kuroda\circ}$. For instance, consider
\[ 
    \bang ((A \cwedge B)_{\Kuroda})_\circ
    \equiv
    \; \bang (A_\Kuroda \cwedge B_\Kuroda)_\circ 
    \equiv
    \; \bang ((\bang (A_\Kuroda)_\circ \cwedge \, \bang (B_\Kuroda)_\circ) \stackrel{\textup{(IH)}}{\lequiv}
    \; \bang (\bang A_{\Kuroda\circ} \cwedge \, \bang B_{\Kuroda\circ})    
\]
and
\[ 
    \bang ((A \lto B)_{\Kuroda})_\circ
    \equiv
    \; \bang (A_\Kuroda \lto B_\Kuroda)_\circ 
    \equiv
    \; \bang (\bang (A_\Kuroda)_\circ \lto (B_\Kuroda)_\circ) \stackrel{\textup{(IH)}}{\lequiv}
    \; \bang (\bang A_{\Kuroda\circ} \lto B_{\Kuroda\circ})    
\]
using again Prop. \ref{bang-equivalences} $(xi)$. The other cases are treated similarly.
\end{proof}

Finally, we consider the composition of the Kuroda translation $(\cdot)^{\Kuroda}$ with  $(\cdot)^*$.

\begin{thm}[$(\cdot)^{\Kuroda*}$ translation] Consider the following translation from $\CLb$ to $\CLLb$ obtained by composing $(\cdot)^{\Kuroda}$ and $(\cdot)^*$:
\eqleft{A^{\Kuroda*}  \pdefin (A^{\Kuroda})^*.}
This composition can be presented directly in a modular (and simpler) way as follows:
\eqleft{A^{\Kuroda*}  \pdefin ~\bang   \whynot   A_{\Kuroda*}}
where $\Tr{A}{\Kuroda*}$ is defined inductively as:
\[
\begin{array}{rclrcl}
    (A\cwedge B)_{\Kuroda*} 
    & \pdefin &   A_{\Kuroda*} \cwedge  B_{\Kuroda*}&
    P_{\Kuroda*} 
    & \pdefin &  \bang P, 
    \textup{~for $P$ atomic} \\[1mm]
    (A \awedge B)_{\Kuroda*} 
    & \pdefin & \bang (A_{\Kuroda*} \awedge B_{\Kuroda*}) \quad & 
    (\forall x A)_{\Kuroda*}
    & \pdefin & \bang \forall x  \whynot  A_{\Kuroda*}\\[1mm]
    (A\oplus B)_{\Kuroda*} 
    & \pdefin &  A_{\Kuroda*} \oplus  B_{\Kuroda*}  &
    (\exists x A)_{\Kuroda*}
    & \pdefin & \exists x A_{\Kuroda*} \\[1mm]
    (A\lto B)_{\Kuroda*}
    & \pdefin &\bang  (A_{\Kuroda*} \lto B_{\Kuroda*}) &
    \Tr{(\bang A)}{\Kuroda*} 
    & \pdefin &  \bang\Tr{A}{\Kuroda*}.
\end{array}
\]
\end{thm}
\begin{proof} We can show by induction on $A$ that $\bang \whynot (A_{\Kuroda})^*$ -- which is $(A^{\Kuroda})^*$ -- is equivalent to $\bang \whynot A_{\Kuroda*}$ -- which is $A^{\Kuroda*}$. In fact, we can prove the stronger result that $(A_{\Kuroda})^*$ is equivalent to $A_{\Kuroda*}$. For instance
\[ 
    ((A \awedge B)_{\Kuroda})^*
    \equiv
    (A_\Kuroda \awedge B_\Kuroda)^* 
    \equiv
    \; \bang ((A_\Kuroda)^* \awedge (B_\Kuroda)^*) \stackrel{\textup{(IH)}}{\lequiv}
    \; \bang (A_{\Kuroda*} \awedge B_{\Kuroda*})    
\]
and
\[ 
    ((A \lto B)_{\Kuroda})^*
    \equiv
    \; \bang ((A_\Kuroda)^* \lto (B_\Kuroda)^*)
    \stackrel{\textup{(IH)}}{\lequiv}
    \; \bang (A_{\Kuroda*} \lto B_{\Kuroda*}).
\]
The other cases are treated similarly.
\end{proof}

\begin{obs} Even though the four translations of $\CLb$ into $\CLLb$ described above are obtained by composing two `optimal' translations, it does not imply that the composition is also optimal. In fact, Theorems 1 - 4 already incorporate some simplifications valid in $\CLLb$. Perhaps the four translations above are capable of even further simplifications (from the inside or outside). Due to the length of the present paper we postpone an investigation into the simplifications of these compositions to a future work. For further readings on `direct' translations of $\CL$ into $\CLL$ see \cite{GirTCS,shi}.
\end{obs}



\bibliographystyle{plain}
\bibliography{dblogic}

@article{ao,
    author = {Arthan, R. and Oliva, P.},
    title = {Double Negation Semantics for Generalisations of {H}eyting Algebras},
    journal = {Studia Logica},
    volume = 109,
    pages = {341-365},
    year = 2021
}

@InProceedings{berger,
author="Berger, U.
and Schwichtenberg, H.",
editor="Leivant, Daniel",
title="Program extraction from classical proofs",
booktitle="Logic and Computational Complexity",
year="1995",
publisher="Springer Berlin Heidelberg",
address="Berlin, Heidelberg",
pages="77--97",

}

@InProceedings{boudard,
author="Boudard, M{\'e}lanie
and Hermant, Olivier",
editor="McMillan, Ken
and Middeldorp, Aart
and Voronkov, Andrei",
title="Polarizing Double-Negation Translations",
booktitle="Logic for Programming, Artificial Intelligence, and Reasoning",
year="2013",
publisher="Springer Berlin Heidelberg",
address="Berlin, Heidelberg",
pages="182--197",

}

@article{Dosen,
    author = {Bo{\v{z}}i{\'c}, M. and Do{\v{s}}en, K.},
    title = {Models for normal intuitionistic modal logics},
    journal = {Studia Logica},
    volume = 43,
    number = 3,
    pages = {217-245}, 
    year = 1983 
}

@techreport{CCP03,
    author = {Chang, Bor-Yuh Evan and Chaudhuri, Kaustuv and Pfenning, Frank},
    title = "A judgmental analysis of linear logic",
    institution = {Carnegie Mellon University},
    year = {April 2003}
}

@book{cosmo,
    author = {Cosmo, R. Di},
    title = {Introduction to Linear Logic},
    publisher = {Notes for the MPRI course},
    year = 1996
}

@article{CFMM16,
    author = {Curien, Pierre-Louis and Fiore, Marcelo and Munch-Maccagnoni, Guillaume},
    title = {A theory of effects and resources: adjunction models and polarised calculi},
    year = {2016},
    issue_date = {January 2016},
    publisher = {Association for Computing Machinery},
    address = {New York, NY, USA},
    volume = {51},
    number = {1},
    issn = {0362-1340},
    url = {https://doi.org/10.1145/2914770.2837652},
    doi = {10.1145/2914770.2837652},
    journal = {SIGPLAN Not.},
    month = jan,
    pages = {44–56},
    numpages = {13},
}

@incollection{DJS95a,
	author = {Vincent Danos and Jean{-}Baptiste Joinet and Harold Schellinx},
	booktitle = {Advances in linear logic},
	editor = {Jean{-}Yves Girard and Yves Lafont and Laurent Regnier},
	pages = {222--211},
	publisher = {Cambridge University Press},
	title = {{LKQ} and {LKT}: Sequent Calculi for Second Order Logic Based Upon Dual Linear Decompositions of Classical Implication},
	year = {1995}
}

@article{DJS95b,
  title={On the linear decoration of intuitionistic derivations},
  author={Vincent Danos and Jean-Baptiste Joinet and Harold Schellinx},
  journal={Archive for Mathematical Logic},
  year={1995},
  volume={33},
  pages={387-412},
  url={https://api.semanticscholar.org/CorpusID:13937107}
}

@article{DJS97, title={A new deconstructive logic: linear logic}, volume={62}, DOI={10.2307/2275572}, number={3}, journal={Journal of Symbolic Logic}, author={Danos, Vincent and Joinet, Jean-Baptiste and Schellinx, Harold}, year={1997}, pages={755–807}}

@article{Far,
    author = {Farahani, H. and Ono, H.},
    title = {Glivenko theorems and negative translations in substructural predicate logics},
    journal = {Archive for Mathematical Logic},
    volume = 51,
    pages = {695-707}, 
    year = 2012
}

@article{ol,
    author = {Ferreira, G. and Oliva, P.},
    title = {On Various Negative Translations},
    journal = {Theoretical Computer Science},
    editors = {Steffen van Bakel, Stefano Berardi, Ulrich Berger},
    year = 2011
}

@incollection{ol2,
    author = {Ferreira, G.  and Oliva, P.},
    title = {On the Relation Between Various Negative Translations},
    booktitle = {Logic, Construction, Computation},
    publisher = {Ontos-Verlag Mathematical Logic Series},
    volume = 3,
    pages = {227-258},
    year = 2012
}

@article{gen1,
    author = {Gentzen, G.},
    title = {Ueber das Verhältnis zwischen intuitionistischer und klassischer Arithmetik},
    journal = {Mathematische Annalen},
    year = 1933 
}

@book{gentzen,
  title        = {The Collected Papers of Gerhard Gentzen},
  editor       = {Szabo, M. E.},
  year         = {1969},
  publisher    = {North-Holland},
  address      = {Amsterdam},
  note         = {Including translations and commentary},
}

@article{GirTCS,
    author = {Girard, J.-Y.},
    title = {Linear Logic},
    journal = {Theoretical Computer Science},
    volume = 50,
    pages = {1-101},
    year = 1987
}

@inproceedings{Lam95,
author = {Lamarche, Francois},
title = {Games Semantics for Full Propositional Linear Logic},
year = {1995},
isbn = {0818670506},
publisher = {IEEE Computer Society},
address = {USA},
booktitle = {Proceedings of the 10th Annual IEEE Symposium on Logic in Computer Science},
pages = {464},
series = {LICS '95}
}

@inproceedings{Lau18,
author = {Laurent, Olivier},
title = {Around Classical and Intuitionistic Linear Logics},
year = {2018},
isbn = {9781450355834},
publisher = {Association for Computing Machinery},
address = {New York, NY, USA},
url = {https://doi.org/10.1145/3209108.3209132},
doi = {10.1145/3209108.3209132},
booktitle = {Proceedings of the 33rd Annual ACM/IEEE Symposium on Logic in Computer Science},
pages = {629–638},
numpages = {10},
location = {Oxford, United Kingdom},
series = {LICS '18}
}

@article{MT10,
  TITLE = {{Resource modalities in tensor logic}},
  AUTHOR = {Melli{\`e}s, Paul-Andr{\'e} and Tabareau, Nicolas},
  URL = {https://hal.science/hal-00339154},
  JOURNAL = {{Annals of Pure and Applied Logic}},
  PUBLISHER = {{Elsevier Masson}},
  VOLUME = {161},
  NUMBER = {5},
  PAGES = {632-653},
  YEAR = {2010},
  MONTH = Feb,
  DOI = {10.1016/j.apal.2009.07.018},
  HAL_ID = {hal-00339154},
  HAL_VERSION = {v2},
}

@article{Sur,
    author = {Glivenko, V. I.},
    title = {Sur quelques points de la logique de {M}. {B}rouwer},
    journal = {Bull. Soc. Math. Belg.},
    volume = 15,
    pages = {183-188},
    year = 1929
}

@article{god,
    author = {G{\"o}del, K.},
    title = {Zur intuitionistischen Arithmetik und Zahlentheorie},
    journal = {Ergebnisse eines Mathematischen Kolloquiums},
    volume = 4,
    pages = {34-38},
    year = 1933
}

@article{heyting,
author="Heyting, A.",
title="Die formalen Regeln der intuitionistischen Logik",
journal="Sitzungsbericht PreuBische Akademie der Wissenschaften Berlin, physikalisch-mathematische Klasse II",
year="1930",
pages="42-56",
URL="https://cir.nii.ac.jp/crid/1571417125840840320"
}

@article{kan,
    author = {Kanovich, M.I. and Okada, M. and Terui, T.},
    title = {Intuitionistic phase semantics is almost classical},
    journal = {Journal of Mathematical Structures in Computer Science},
    volume = 16,
    number = {1},
    pages = {67-86},
    year = 2006
}

@book{kohlenbach,
  series = {Springer Monographs in Mathematics},
  title = {Applied Proof Theory: Proof Interpretations and their Use in
           Mathematics},
  publisher = {Springer Verlag},
  author = {Kohlenbach, Ulrich},
  pages = {xx+536 pp.},
  year = 2008,
}

@article{kol,
    author = {Kolmogorov, A. N.},
    title = {On the principle of the excluded middle (russian)},
    journal = {Mat. Sb.},
    volume = 32,
    pages = {646-667},
    year = 1925
}

@article{kri,
    author = {Krivine, J.},
    title = {Opérateurs de mise en mémoire et traduction de {G}{\"{o}}del},
    journal = {Arch. Math. Logic},
    volume = 30,
    issue = 4,
    pages = {241-267},
    year = 1990
}

@article{kur,
    author = {Kuroda, S.},
    title = {Intuitionistische Untersuchungen der formalistischen Logik},
    journal = {Nagoya Mathematical Journal},
    volume = 3,
    pages = {35–47},
    year = 1951,
}

@inproceedings{lau2,
  TITLE = {About Translations of Classical Logic into Polarized Linear Logic},
  BOOKTITLE = {Annual Symposium on Logic in Computer Science},
  AUTHOR = {Laurent, Olivier and Regnier, Laurent},
  URL = {https://hal.science/hal-00009139},
  EDITOR = {Phokion G. Kolaitis},
  PUBLISHER = {{IEEE}},
  PAGES = {11-20},
  YEAR = {2003},
  DOI = {10.1109/LICS.2003.1210040},
  HAL_ID = {hal-00009139},
  HAL_VERSION = {v1},
}

@incollection{Litak,
    author = {Litak, T. and Polzer, M. and Rabenstein, U.},
    title = {Negative Translations and Normal Modality},
    booktitle = {2nd International Conference on Formal Structures for Computation and Deduction (FSCD 2017). Leibniz International Proceedings in Informatics (LIPIcs), Vol- 84},
    publisher = {Schloss Dagstuhl – Leibniz-Zentrum für Informatik},
    pages = {27:1-27:18},
    year = 2017
}

@book{mints,
	address = {New York},
	author = {G.E. Mints},
	editor = {},
	publisher = {Kluwer Academic / Plenum Publishers},
	title = {A Short Introduction to Intuitionistic Logic},
	year = {2000}
}

@article{Ecumenic,
    author = {Pereira, L.C. and Pimentel, E. and de Paiva, V.},
    title = {Translations and {P}rawitz’s
Ecumenical System},
    journal = {Studia Logica},
    volume = {113},
    pages = {523-538},
    year = 2025
}

@book{lecture,
    author = {Heine S{\o}rensen, M. and Urzyczyn, P.},
    title = {Lectures on the Curry-Howard Isomorphism},
    publisher = {Elsevier},
    volume = {149},
    year = 2006
}

@article{Sch91,
    author = {Schellinx, H.},
    title = {Some Syntactical Observations on Linear Logic},
    journal = {Journal of Logic and Computation},
    volume = {1},
    number = {4},
    pages = {537-559},
    year = {1991},
    month = {09},
    issn = {0955-792X},
    doi = {10.1093/logcom/1.4.537},
    url = {https://doi.org/10.1093/logcom/1.4.537},
    eprint = {https://academic.oup.com/logcom/article-pdf/1/4/537/3817158/1-4-537.pdf},
}

@article{Shi,
    author = {Shirahata, M.},
    title = {The Dialectica Interpretation of First-Order Classical Affine Logic},
    journal = {Theory and Applications of Categories},
    volume = {17},
    issue = 4,
    pages = {49-79},
    year = 2006
}

@book{lec,
    author = {Troelstra, A.S.},
    title = {Lectures on Linear Logic},
    publisher = {Lecture Notes No. 29, CSLI},
    year = 1992
}

@unpublished{prover9-mace4,
    author = "W. McCune",
    title = "Prover9 and {M}ace4",
    note = "{\url{http://www.cs.unm.edu/~mccune/prover9/}}",
    year = "2005--2010"
}

@book{basic,
    author = {Troelstra, A.S. and Schwichtenberg, H.},
    title = {Basic Proof Theory},
    publisher = {CUP Cambridge},
    year = 2000
}

@incollection{cons,
    author = {Troelstra, A.S. and van Dalen, D.},
    title = {Constructivism in Mathematics: An Introduction},
    volume = 1,
    booktitle = {Studies in Logic and the Foundations of Mathematics},
    publisher = {Elsevier Science},
    year = 1988
}

\end{document}